\documentclass[12pt]{amsart}
\usepackage{amsmath,amsthm,amssymb,mathrsfs,amscd,amsthm,amscd,
amsfonts,graphicx}

\makeindex
\usepackage[cmtip,all]{xy}  
\usepackage{texdraw}
\usepackage{eucal}
\usepackage[T1]{fontenc}
\usepackage[utf8]{inputenc}
\usepackage{indentfirst,url,xypic}
\usepackage{lipsum}
\usepackage{faktor}
\usepackage[all]{xy}
\usepackage{url}
\usepackage[colorlinks,plainpages]{hyperref}
\usepackage{verbatim}
\usepackage{fancyvrb}
\usepackage{listings}
\usepackage{mathtools}
\usepackage{algorithm}
\usepackage[noend]{algorithmic}
\usepackage{enumitem}  
\usepackage[colorlinks,plainpages]{hyperref}
\hypersetup{
	colorlinks=true,
	linkcolor=blue,
	filecolor=magenta,      
	urlcolor=cyan,
	citecolor=blue
}

\newcommand{\alp}{{\alpha}}
\newcommand{\bet}{{\beta}}

\newcommand{\bi}{\boldsymbol{i}}

\newcommand{\bnu}{\boldsymbol{\nu}}
\newcommand{\bo}{\boldsymbol{0}}

\newcommand{\bphi}{\boldsymbol{\phi}}

\newcommand{\bt}{\boldsymbol{t}}

\newcommand{\bx}{\boldsymbol{x}}

\newcommand{\del}{{\delta}}
\newcommand{\Del}{{\Delta}}
\newcommand{\eps}{{\varepsilon}}

\newcommand{\lra}{\longrightarrow}

\newcommand{\pt}{\text{pt}}

\newcommand{\sig}{{\sigma}}

\DeclareMathOperator{\id}{id}

\DeclareMathOperator{\codim}{codim}

\DeclareMathOperator{\coker}{coker}

\newtheorem{definition}{Definition}[section]
\newtheorem{theorem}[definition]{Theorem}
\newtheorem{remark}[definition]{Remark}
\newtheorem{proposition}[definition]{Proposition }
\newtheorem{corollary}[definition]{Corollary}
\newtheorem{conjecture}[definition]{Conjecture}
\newtheorem{problem}[definition]{Problem}
\newtheorem{question}[definition]{Question}
\newtheorem{example}[definition]{Example}

\newtheorem{Proof}[definition]{Proof}
\newtheorem{Proof*}{Proof}
\newtheorem{lemma}[definition]{Lemma}

\newcommand{\Z}{{\mathbb Z}}

\newcommand{\C}{{\mathbb C}}

\renewcommand{\P}{{\mathbb P}}

\newcommand{\kd}{{\mathcal D}}

\newcommand{\ki}{{\mathcal I}}

\newcommand{\km}{{\mathcal M}}
\newcommand{\kn}{{\mathcal N}}
\newcommand{\ko}{{\mathcal O}}

\newcommand{\ks}{{\mathcal S}}

\newcommand{\fm}{\mathfrak{m}}

\DeclareMathAlphabet{\mathsc}{U}{rsfs}{m}{n}

\newcommand{\sC}{\mathsc{C}}

\newcommand{\sH}{\mathsc{H}}

\newcommand{\sS}{\mathsc{S}}

\newcommand{\sX}{\mathsc{X}}
\newcommand{\sY}{\mathsc{Y}}

\DeclareMathOperator{\Sing}{Sing}

\DeclareMathOperator{\Def}{\kd\!\!\:\it{ef}}

\newcommand{\uDef}{\underline{\Def}\!\,}

\DeclareMathOperator{\edim}{edim}
\DeclareMathOperator{\Ext}{Ext}

\DeclareMathOperator{\Hom}{Hom}

\DeclareMathOperator{\im}{Im}

\DeclareMathOperator{\Ker}{Ker}

\DeclareMathOperator{\kDer}{\text{$\kd$\it{er}}\,}

\DeclareMathOperator{\kHom}{\text{$\sH$\!\!\:\it{om}}}

\newcommand{\Kos}{\text{Kos}}

\renewcommand{\mod}{\text{mod } }
\DeclareMathOperator{\Mor}{Mor}

\DeclareMathOperator{\ob}{ob}

\newcommand{\Rel}{\text{Rel}}

\DeclareMathOperator{\Sets}{\ks\it{ets}}

\newcommand{\Teps}{T_{\varepsilon}}

\begin{document}

\title{Deformation and Smoothing of Singularities}
\author{Gert-Martin Greuel}


\maketitle

\begin{abstract}
We give a survey on some aspects of deformations of isolated singularities. In addition to the presentation of the general theory, we report on the question of the smoothability of a singularity and on relations between different invariants, such as the Milnor number, the Tjurina number, and the dimension of a smoothing component.
\end{abstract}

\renewcommand{\contentsname}{Table of Contents}
\tableofcontents




\section*{Introduction}
\vspace{1.0cm}


This is a survey on some aspects of deformations of isolated singularities. In addition to the presentation of the general theory, we report on the question of the smoothability of a singularity and on relations between different invariants, such as e.g. the Milnor number, the Tjurina number, and the dimension of a smoothing component.

In the first chapter we give an overview on the theory of deformations of complex space germs.  Although we use the language of functors for precise statements, we provide also explicit descriptions in terms of the defining equations. 
We give almost no proofs but for every statement we give precise references to sources that contain more details, including proofs, for further reading. Since we confine ourselves to the deformation theory of isolated singularities we avoid the almost unmanageable field of far-reaching generalizations. Thus we give a compact presentation that nevertheless contains all essential fundamental results.

The second chapter is devoted to the study of the nearby fiber, also called Milnor fibre, of a small deformation and provides a short overview of the historically most relevant results on rigidity and smoothability. Moreover, we discuss known results and conjectures about the relationship between the dimension of a smoothing component and the topology of the Milnor fibre over this component.
An important question is which invariants of a smoothing are independent of this and depend only on the singulariy.
The main results concern complete intersections, as well as curve and surface singularities, which we treat separately in different subsections. In addition to 
known results, we also discuss open problems and conjectures.
As a rule, we give no proofs, but we sketch them in cases where the method is particularly interesting. For all results we give exact references.
\medskip


{\bf Acknowlegment:} I would like to thank Helmut Hamm and Toni Iarrobino for useful comments.
\section{Deformation Theory}

\subsection{Deformations of Complex Germs}
\setcounter{equation}{0}

We give an overview of the deformation theory of isolated singularities of complex
space germs. The concepts and theorems for this case may serve  
as a prototype for deformations of other objects, such as deformations of mappings or,
more general, of deformations of diagrams.
For the theory of complex spaces \index{complex space} and their morphisms, also called holomorphic or analytic maps, we refer to \cite{GR84} and to \cite{GLS07}. Good  references for deformations of algebraic schemes are the books
\cite{Ha10} and \cite{Se06}.
\medskip

A \index{complex space!pointed} {\em pointed complex space} is a pair $(X,x)$ consisting of a complex space $X$ and a point $x\in X$. A morphism $f :(X,x) \to (Y,y)$ of pointed complex spaces is a morphism $f : X \to Y$ of complex spaces such that $f(x) = y$. The structure sheaf of $X$ is denoted by $\ko_X$, the (analytic) local ring by  $\ko_{X,x}$ with maxiaml ideal $\fm$, and the induced map of local rings 
by $f^\sharp: \ko_{Y,y} \to \ko_{X,x}$.
A {\em complex (space) germ} is a pointed complex spaces where the morphisms are equivalence classes of morphisms of pointed complex spaces. 
Two morphisms 
$f$ resp. $g$ \ from $(X,x)$ to $(Y,y)$ defined in some open neighbourhood  $U$ resp. $V$ of $x$ are called equivalent, if they coincide on a neighbourhood $W \subset U\cap V$ of $x$. A {\em singularity} \index{singularity} is nothing but a complex space germ.

If $U \subset X$ is an open subset of the complex space $X$, then 
$\Gamma(U,\ko_X)$ denotes the $\C$-algebra of holomorphic functions on $X$.
If $I = \langle f_1,..., f_k\rangle $ is an ideal in $\Gamma(U,\ko_X)$, generated by $f_1,..., f_k$, we denote by $V(I) = V(f_1,..., f_k) $ the (closed) complex subspace of $U$, being as topological space $\{x \in U | f_1(x) =...= f_k = 0\}$ with 
structure sheaf $\ko_U / I \ko_U$. 
 \medskip

 \begin{definition}\label{def:deform}
Let $(X,x)$ and $(S,s)$ be complex space germs. A {\em deformation of
$(X,x)$ over $(S,s)$}\index{deformation!of complex
germs}\index{germ!deformation}\index{deformation} consists of a
flat\index{flat}\index{flatness} morphism 
$\phi\colon({\sX},x)\to(S,s)$ of complex germs together with an
isomorphism $(X,x)\xrightarrow{\text {$\cong$}} (\sX_s,x)$. 
$(\sX,x)$ is called the {\em total space}, $(S,s)$ the {\em base
space}\index{deformation!total
space of}\index{deformation!base
space of}\index{base!space}\index{total!space}, and
$(\sX_s,x):=(\phi^{-1}(s),x)$ or $(X,x)$ the {\em special fibre}\index{special  fibre}\index{fibre!special} of the deformation. 
\end{definition}

We can write a deformation as a Cartesian diagram
\begin{equation*}
\xymatrix@C=16pt@R=15pt@M=5pt{
(X,x) \ar@{}[dr]|{\Box}\ar@{^{(}->}[r]^-{i} \ar[d] & (\sX,x)
\ar[d]^-{\phi\text{ flat}}\\
\{ \pt\} \ar@{^{(}->}[r] & (S,s)}, 
\end{equation*}
where $i$ is a closed embedding mapping $(X,x)$ isomorphically onto
\mbox{$(\sX_s,x)$} and $\{\pt\}$ the reduced point considered as a
complex space germ with local ring $\C$.
We denote a deformation by
$$ (i,\phi)\colon(X,x)\stackrel{i}{\hookrightarrow}
(\sX,x)\stackrel{\phi}{\to}(S,s)\,,$$
or simply by \mbox{$\phi:(\sX,x)\to(S,s)$} in order to shorten notation.
Note that the closed embedding $i$ is part of the data and identifies $(\sX_s,x)$ and $(X,x)$. Thus, if $(\sX'\!,x)\to (S,s)$ is another deformation of $(X,x)$,
 then there is a
{\em unique\/} isomorphism of germs $(\sX_s,x)\cong (\sX'_s,x)$.
\medskip

The essential point here is that $\phi$ is {\em flat at $x$}\index{flat at a point},
 that is, $\ko_{\sX,x}$ is a flat $\ko_{S,s}$-module via the induced morphism
\mbox{$\phi_x^\sharp:\ko_{S,s}\to \ko_{\sX,x}$}.
A well known theorem of Frisch (cf. \cite{Fr67})\index{Frisch's theorem} 
says that for a morphism
\mbox{$\phi:\sX\to S$} of complex spaces the set of points in $\sX$
where $\phi$ is flat is analytically open.
Hence, a sufficiently small representative $\phi:\sX\to S$ of the germ
$\phi$ is everywhere flat and, since flatness implies 
$\dim(\sX_s,x)=\dim(\sX,x)-\dim(S,s)$,
we have $\dim(\sX_t,y)=\dim(\sX_s,x)$ for all
$t\in S$ and all $y\in \sX_t$ if $\sX$ and $S$
are pure dimensional. Another important theorem is due to Douady \cite{Do66}, saying that every flat morphism $\phi:\sX\to S$ of complex spaces is
  open, that is, it maps open sets in $\sX$ to open sets in $S$.
 An important example of flat morphisms are projections: If $X, T$ are 
  complex spaces then the projection $X \times T \to T$ is flat (c.f. \cite[Corollary I.1.88]{GLS07}).


\smallskip
A typical example of a non-flat morphism is the projection $(\C^2,\bo)\supset V(xy)\to(\C,0), (x,y) \mapsto x$, since the fibre-dimension jumps (the dimension of the special fibre is 1 and of the other fibres is 0).
\medskip

The following  theoretically and computationally useful criterion for flatness is due to Grothendieck (for a proof see e.g. \cite[Proposition I.1.91]{GLS07}).

\begin{proposition}[Flatness by relations]\label{prop:flatbyrel}\index{flatness!criterion!by relations} 
Let \mbox{$I=\langle f_1,\dots,f_k\rangle\subset\ko_{\C^n\!,\bo}$} be an ideal,
\mbox{$(S,s)$} a complex space germ and \mbox{$\widetilde{I}=\langle
  F_1,\dots,F_k\rangle\subset\ko_{\C^n\!\times S,(\bo,s)}$} a lifting of $I$,
i.e., $F_i$ is a preimage of $f_i$ under the surjection
$$ \ko_{\C^n\!\times S,(\bo,s)}\twoheadrightarrow\ko_{\C^n\!\times
  S,(\bo,s)}\otimes_{\ko_{S,s}}\C =\ko_{\C^n\!,\bo}\,. $$
Then the following are equivalent:
\begin{enumerate}[leftmargin=*]
\itemsep3pt
\item[(a)] $\ko_{\C^n\!\times S,(\bo,s)}/\widetilde{I}$ is
  $\ko_{S,s}$-flat;
\item[(b)] any relation $(r_1,...,r_k)$ among $f_1,...,f_k$
  lifts to a relation $(R_1,...,R_k)$ among $F_1,...,F_k$. That is, for each $(r_1,...,r_k)$ satisfying 
  $$ \sum_{i=1}^k r_if_i = 0\,,\quad r_i\in\ko_{\C^n\!,\bo}\,, $$
 there exists $(R_1,...,R_k)$ such that
  $$ \sum_{i=1}^k R_iF_i = 0\,,\ \text{with }\: R_i\in\ko_{\C^n\!\times
    S,(\bo,s)} $$ 
  and the image of $R_i$ in \mbox{$\ko_{\C^n\!,\bo}$} is $r_i$;
\item[(c)] any free resolution of \mbox{$\ko_{\C^n\!,\bo}/I$}
  $$
    \ldots\to\ko_{\C^n\!,\bo}^{p_2}\to\ko_{\C^n\!,\bo}^{p_1}\to\ko_{\C^n\!,\bo}
    \to\ko_{\C^n\!,\bo}/I\to 0
  $$
  lifts to a free resolution of $\ko_{\C^n\!\times S,(\bo,s)}/\widetilde{I}$,
  $$
    \ldots\to\ko_{\C^n\!\times S,(\bo,s)}^{p_2}\to\ko_{\C^n\!\times
      S,(\bo,s)}^{p_1}\to\ko_{\C^n\!\times S,(\bo,s)}
    \to\ko_{\C^n\!\times S,(\bo,s)}/\widetilde{I}\to 0\,.
  $$
  That is, the latter sequence tensored with $\otimes_{\ko_{S,s}}\C$
  yields the first sequence.
\end{enumerate}
\end{proposition}


\begin{remark}\label{ex.flat}
Let us recall some geometric consequences
  of flatness, for an algebraic proof see e.g. \cite[Theorem B.8.11 and B.8.13]{GLS07} and \cite[Theorem 15.1]{Ma86}.
\begin{enumerate}[leftmargin=*]
\item  $\phi=(\phi_1,\dots,\phi_k):(\sX,x)\to (\C^k\!,\bo)$ is
  flat iff $\phi_1,\dots,\phi_k$ is a regular sequence.
\item If $(\sX,x)$ is Cohen-Macaulay, then
$\phi_1,\dots,\phi_k\in \fm\subset \ko_{\sX,x}$ is a
  regular sequence iff $\dim
    \ko_{\sX,x}/\langle \phi_1,\dots,\phi_k\rangle
     = \dim  (\sX,x)-k$. 
\item In particular, \mbox{$\phi:(\C^m\!,\bo)\to
    (\C^k\!,\bo)$} is flat iff $\dim\bigl(\phi^{-1}(\bo),\bo\bigr)=m-k$.
    If this holds $(X,\bo):=(f^{-1}(\bo),\bo)$ is called a {\em complete
intersection}\index{complete!intersection!singularity} and
$(i,f):(X,\bo)\subset(\C^m\!,\bo)\to(\C^k\!,\bo)$ is a deformation of
$(X,\bo)$ over \mbox{$(\C^k\!,\bo)$}. If $k=1$ then $(X,\bo)$ is called a
{\em hypersurface singularity}. \index{hypersurface singularity}
\end{enumerate}
\end{remark}

Note that smooth germs, {\em hypersurface} and {\em complete intersection singularities}\,\footnote{$(X,x) \subset (\C^N,x)$ is a complete intersection if the minimal number of generators of its ideal $I(X,x) \subset \ko_{\C^N,x}$ is $N-\dim(X,x).$ $(X,x)$ is a hypersurface singularity if $\dim(X,x)=N-1$.}
\index{singularity!hypersurface}\index{singularity!complete intersection}\index{complete!intersection}
reduced curve singularities and normal surface singularities are Cohen-Macaulay.

\begin{definition}
\label{def:morph of def}
Given two deformations $(i,\phi) \colon
  (X,x)\hookrightarrow(\sX,x)\to(S,s)$ and $(i',\phi')
  \colon (X,x)\hookrightarrow(\sX',x')\to(S',s')$, of $(X,x)$ over
$(S,s)$ and $(S',s')$, respectively. A {\em morphism
  of deformations}\index{morphism!of deformations}\index{deformation!morphism
  of} 
from $(i,\phi)$ to $(i',\phi')$ is a morphism of the diagram after Definition \ref{def:deform}
being the 
identity on $(X,x)\to \{\pt\}$. Hence, it consists of two morphisms
\mbox{$(\psi,\varphi)$} such 
that the following diagram commutes
$$\xymatrix@C=18pt@R=15pt@M=4pt{
& (X,x) \ar@{_{(}->}[dl]_-{i'} \ar@{^{(}->}[dr]^-{i} \\
(\sX',x') \ar[rr]^{\psi} \ar[d]_-{\phi'} & & (\sX,x)
\ar[d]^-{\phi}\\
(S',s')\ar[rr]^-{\varphi} & & (S,s)}. $$
Two deformations over the {\em same} base space $(S,s)$ are {\em
isomorphic}\index{isomorphic!deformations}\index{deformation!isomorphic} 
if there exists a morphism \mbox{$(\psi,\varphi)$} with $\psi$ an
isomorphism and \mbox{$\varphi$} the identity map.
\end{definition}

It is easy to see that deformations of $(X,x)$ form a
category. Usually one considers the (non-full) subcategory of
deformations of $(X,x)$ over a fixed base space $(S,s)$ and morphisms
$(\psi,\varphi)$ with $\varphi=\id_{(S,s)}$. 
The following lemma implies that this category is a {\em groupoid}, i.e., all morphisms are automatically isomorphims (see e.g. \cite[Lemma I.1.86]{GLS07} for a proof).

\begin{lemma} \label{lem:deformation:1}
Let
$$\xymatrix@C=16pt@R=15pt@M=2pt{
(X,x) \ar[rr]^-{f}\ar[dr]_-{\phi} & & (Y,y)
\ar[dl]^-{\psi}\\ & (S,s)}$$
be a commutative diagram of complex germs with $\phi$ flat. Then $f$ is an
isomorphism iff $f$ induces an isomorphism of the special fibres,
$ f\colon
(\phi^{-1}(s),x)\stackrel{\cong}{\longrightarrow}(\psi^{-1}(s),y)\,.
$
\end{lemma}
\medskip

\noindent
We introduce now the concept of induced
deformations, which give rise, in a natural way, to  
morphisms between deformations over different base spaces.

Let \mbox{$(X,x)\hookrightarrow(\sX,x)\stackrel{\phi}{\to}(S,s)$} be a
deformation of the complex space germ $(X,x)$ and
\mbox{$\varphi\colon(T,t)\to(S,s)$} a 
morphism of germs. Then the fibre product 
is the following commutative diagram of germs
$$
\xymatrix@!C=40pt@R=15pt@M=5pt{
& (X,x) \ar@{_{(}->}[dl]_-{\varphi^{\ast}i} \ar@{^{(}->}[dr]^-{i} \\
(\sX,x)\times_{(S,s)}(T,t) \ar[rr]^-{\widetilde{\varphi}}
\ar[d]_-{\varphi^{\ast}\phi} & &  (\sX,x)
\ar[d]^-{\phi}\\
(T,t)\ar[rr]^-{\varphi} & & (S,s)} $$
where $\varphi^{\ast}\phi$, resp.\ $\widetilde{\varphi}$, are induced by
the second, resp.\ first, projection, and
$ \varphi^{\ast}i=\big(\widetilde{\varphi}\big|_
{(\varphi^{\ast}\phi)^{-1}(t)}\big)^{-1}\circ i\,. $

\begin{definition}
\label{def:induced def}
We denote \mbox{$(\sX,x)\times_{(S,s)}(T,t)$} by
\mbox{$\varphi^{\ast}(\sX,x)$}\index{$phistar$@$\varphi^{\ast}(\sX,x)$}
and call
$$
\varphi^{\ast}(i,\phi):= (\varphi^{\ast}i,\varphi^{\ast}\phi)\colon
(X,x)\stackrel{\varphi^{\ast}i}{\hookrightarrow}
\varphi^{\ast}(\sX,x)\stackrel{\varphi^{\ast}\phi}{\longrightarrow}
(T,t)$$
the {\em deformation induced by $\varphi$ from
  $(i,\phi)$}\index{induced deformation}\index{deformation!induced},
or just the {\em induced deformation} or {\em
  pull-back\/}\index{pull-back}; $\varphi$ is called the {\em base 
  change map}.\index{base!change}
\end{definition}

Since flatness is preserved under base change (c.f. \cite [Proposition I.187]{GLS07}), $\varphi^{\ast}\phi$ is
flat. Hence, $\varphi^{\ast}(i,\phi)$ is indeed a
deformation of $(X,x)$ over $(T,t)$, and
\mbox{$(\widetilde{\varphi},\varphi)$} is a morphism from
\mbox{$(i,\phi)$} to \mbox{$(\varphi^{\ast}i,\varphi^{\ast}\phi)$},
and $\varphi^{\ast}$ preserves isomorphisms.
A typical example of an induced deformation is the restriction to a subspace in
the parameter space $(S,s)$.

\medskip

We introduce the following notations.

\begin{definition}\label{def:Def_Xx}
  Let $(X,x)$ be a complex space germ.
  \begin{enumerate}[leftmargin=*]
  \itemsep3pt
  \item[(1)] $\Def_{(X,x)}$\index{$defxx$@$\Def_{(X,x)}$}
  denotes the {\em category of
  deformations of $(X,x)$}, with morphisms
as defined in Definition \ref{def:morph of def}.
  \item[(2)] $\Def_{(X,x)}(S,s)$
  denotes the {\em category of
  deformations of $(X,x)$ over $(S,s)$}, whose morphisms satisfy
  \mbox{$\varphi=\id_{(S,s)}$}. 
  \item[(3)] $\uDef_{(X,x)}(S,s)$
  \index{$defxx2$@$\uDef_{(X,x)}$} denotes the set of
  {\em isomorphism classes of deformations $(i,\phi)$ of $(X,x)$ over
  $(S,s)$}. 

\noindent   
   For a morphism of complex germs $\varphi\colon(T,t)\to(S,s)$, the
  pull-back $\varphi^{\ast}(i,\phi)$ is a deformation of $(X,x)$ over $(T,t)$,
  inducing  a map $\uDef_{(X,x)}(S,s) \to \uDef_{(X,x)}(T,t)$.
%
 It follows that
 $$ \uDef_{(X,x)}\colon \text{({\it complex
    germs})}\longrightarrow \Sets\,,$$ 
$(S,s)\mapsto\uDef_{(X,x)}(S,s) $
  is a functor, the {\em deformation
    functor\/}\index{deformation!functor} or the {\em
    functor of isomorphism classes of deformations\/} of $(X,x)$.
\end{enumerate}
\end{definition}
\bigskip

\subsection{Embedded
  Deformations and Unfoldings}\index{embedded!deformation}\index{deformation!embedded} \index{unfolding}
\setcounter{equation}{0}

\noindent
This section aims at describing the somewhat abstract definitions of
the preceding section in more concrete terms, that is, in terms of
defining equations and relations. Moreover, we derive a
characterization of flatness via lifting of relations.
\medskip

Let us first recall the notion of unfoldings of a hypersurface singularities 
and explain its relation to deformations.
Given $f\in\C\{x_1,...,x_n\}$, $f(\bo)=0$, an {\em
  unfolding\/}\index{unfolding} of $f$ is a power series
$F\in\C\{x_1,...,x_n,t_1,...,t_k\}$ with
$F(\bx,\bo)=f(\bx)$, that is, 
$$F(\bx,\bt)=f(\bx)+\sum_{|\bnu|\geq 1} g_{\bnu}(\bx)\bt^{\bnu}\,.$$
We identify the power series $f$ and $F$ with the holomorphic map
germs $$f\colon(\C^n\!,\bo)\to(\C,0)\,,\quad
F\colon(\C^n\!\times\C^k,\bo)\to(\C,0)\,.$$ 
Then $F$ induces a deformation of the hypersurface singularity $(X,\bo)=(f^{-1}(0),\bo)$ over $\C^k$ in the
following way 
$$\xymatrix@C=16pt@R=15pt@M=5pt{ 
(X,\bo) \ar[d]\ar@{^{(}->}[r]^-{i} & (\sX,\bo)\ar[d]^{\phi=pr_2|_{(\sX,\bo)}}
& \hspace*{-20pt}:=(F^{-1}(0),\bo) & 
\hspace*{-20pt}\subset(\C^n\!\times\C^k,\bo)\\
\{ 0\} \ar@{^{(}->}[r] & (\C^k,\bo)} $$
where $i$ is the inclusion and $\phi$ the restriction of the second
projection.
By Remark \ref{ex.flat} $(i,\phi)$ is
a deformation of $(X,\bo)$.  Indeed, each deformation of
\mbox{$(X,\bo)=(f^{-1}(0),\bo)$} over some \mbox{$(\C^k\!,\bo)$} is induced by
an unfolding of $f$. This follows from Corollary \ref{cor:defs can be embedded}
below.

More generally, we have the following important result.

\begin{proposition}[Embedding of a morphism]\label{prop:emb def}
 Given a Cartesian diagram of complex space germs
$$\xymatrix@C=24pt@R=15pt@M=5pt{
(X_0,x)
\ar@{^{(}->}[r]\ar[d]_-{f_0}\ar@{}[dr]|{\Box}
& (X,x) \ar[d]^-{f} \\
(S_0,s) \ar@{^{(}->}[r] & (S,s) \, ,} $$
where the horizontal maps are closed embeddings. Assume that $f_0$
factors as
$$ (X_0,x)\stackrel{i_0}{\hookrightarrow}(\C^n\!,\bo)\times(S_0,s)
\stackrel{p_0}{\to}(S_0,s) $$
with $i_0$ a closed embedding and $p_0$ the second projection.\footnote{In this
  situation, we call $f_0$ an {\em embedding over 
    $(S_0,s)$}.}  Then
there exists a Cartesian diagram
\begin{eqnarray}
\label{eq.1}
\xymatrix@R=20pt@M=5pt{
(X_0,x)
\ar@{^{(}->}[r]\ar@{^{(}->}[d]_-{i_0}\ar@/_2pc/@<-4ex>[dd]_{f_0}\ar@{}[dr]|{\Box} &
(X,x)\ar@{^{(}->}[d]^i\ar@/^2pc/@<4ex>[dd]^f \\
(\C^n\!,\bo)\times(S_0,s) \ar@{^{(}->}[r]\ar[d]_-{p_0}\ar@{}[dr]|{\Box} &
(\C^n\!,\bo)\times(S,s)\ar[d]^-{p}\\ 
(S_0,s) \ar@{^{(}->}[r] & (S,s)}
\end{eqnarray}
with $i$ a closed embedding and $p$ the second projection. That is,
the embedding of $f_0$ over $(S_0,s)$ extends to an embedding of $f$
over $(S,s)$.
\end{proposition}

\noindent
Note that we do not require that $f_0$ or $f$ are flat. The proof is not difficult, see \cite[Proposition]{GLS07}.\\

Applying Proposition \ref{prop:emb def} to a deformation of $(X,x)$ 
we get
\begin{corollary}[Embedding of a deformation] 
\label{cor:defs can be embedded}
 Let $(X,\bo)\subset(\C^n\!,\bo)$ be a closed subgerm. Then any
deformation of $(X,\bo)$,
$$(i,\phi) : (X,\bo){\hookrightarrow}(\sX,x){\to}(S,s),$$
can be embedded. That is, there exists a Cartesian diagram
$$\xymatrix@C=10pt@R=7pt@M=5pt{
(X,\bo)\ \ar@{}[ddrr]|{\Box}\ar@{^{(}->}[rr]^-{i}\ar@{^{(}->}[dd] && (\sX,x)\ar@{^{(}->}[dd]^-{J}
\ar@/^3pc/@<3ex>[dddd]^{\phi}\\
\\
(\C^n\!,\bo)\ \ar@{}[ddrr]|{\Box}\ar@{^{(}->}[rr]^-{j}\ar[dd] && (\C^n,\bo)\times(S,s)\ar[dd]^-{p}\\
\\
\{ s\} \ar@{^{(}->}[rr] && (S,s) }$$
where $J$ is a closed embedding, $p$ is the second projection and $j$
the first inclusion. 

In particular, the embedding
dimension\index{embedding!dimension!is semicontinuous} is semicontinuous under
deformations, that is, \mbox{$\edim 
  \bigl(\phi^{-1}(\phi(y)),y\bigr) \leq 
  \edim (X,\bo)$}, for all $y$ in $\sX$ sufficiently close to $x$.
\end{corollary}

\begin{remark}\label{rm.explicit}
We get the following {\em explicit description of a deformation:}

Any deformation $(i,\phi)\colon (X,\bo)\hookrightarrow(\sX,x)\to(S,s)$ of 
$(X,\bo)$ can be assumed to be given as follows: Let
\mbox{$I_{X,\bo}=\langle f_1,\dots,f_k\rangle \subset \ko_{\C^n\!,\bo}$}
be the ideal of \mbox{$(X,\bo)\subset (\C^n\!,\bo)$}. The embedding 
of the total space of the deformation of $(X,\bo)$ is given as
  $$(\sX,x)=V(F_1,\dots,F_k)\xhookrightarrow{\text{J}}\C^n\!\times S,(\bo,s))\,,$$ 
with $\ko_{\sX,x}=\ko_{\C^n\!\times S,(\bo,s)}/I_{\sX,x}$, 
$ I_{\sX,x}=\langle F_1,\dots,F_k\rangle\subset\ko_{\C^n\!\times
  S,(\bo,s)}$ and $f_i$ being the image of $F_i$ in $\ko_{\C^n\!\times
S,(\bo,s)}/\fm_{S,s}=\ko_{\C^n\!,\bo}$. 
Then $(X,0) = V(f_1,...,f_k) \xhookrightarrow{\it{i'}} (\sX,x)$ and setting $\phi' = p\circ J$, $p$ the second projection,  we get the deformation $(i' \phi')$, which coincides $(i,\phi)$ up to isomorphism.

Furthermore, let \mbox{$(S,s)\subset(\C^r\!,\bo)$} and denote the coordinates
of $\C^n$ by \mbox{$\bx=(x_1,\dots,x_n)$} and those of $\C^r$ by
\mbox{$\bt=(t_1,\dots,t_r)$}. Then \mbox{$f_i=F_i|_{(\C^n,\bo)}$} and,
hence, $F_i$ is 
of the form\,\footnote{That a system of generators for $I_{\sX,x}$ can be written in
this form follows from the fact that \mbox{$\fm_{S,s}I_{\sX,x} =
  \fm_{S,s}\ko_{\C^n\!\times S,(\bo,s)}\cap I_{\sX,x}$}, which is a
consequence of flatness.}
\begin{equation*}
F_i(\bx,\bt)=f_i(\bx)+\sum_{j=1}^r t_jg_{ij}(\bx,\bt)\,, \quad
g_{ij}\in\ko_{\C^n\!\times\C^r,\bo}\,,
\end{equation*}
that is, $F_i$ is an unfolding of
$f_i$. 
\end{remark}

In general, the $F_i$ as above do not define a deformation, since the flatness condition is not fulfilled. However, if $(X,\bo)$ is  an $(n-k)$-dimensional complete intersection, flatness is automatic.

\begin{proposition}
\label{prop:unf vs def}
Let \mbox{$(X,\bo)\subset(\C^n\!,\bo)$} be a complete intersection, and
let $f_1,\dots,f_k$ be a minimal set of generators of the ideal of $(X,\bo)$ in
$\ko_{\C^n,\bo}$. Then, for any complex germ $(S,s)$ and any lifting
\mbox{$F_i\in\ko_{\C^n\!\times S,(\bo,s)}$} of $f_i$, $i=1,\dots,k$ (i.e., $F_i$ is 
of the form as in Remark (\ref{rm.explicit})), the diagram
$$ (X,\bo)\hookrightarrow(\sX,x)\stackrel{p}{\to}(S,s) $$
with $(\sX,x)=V(F_1,...,F_k )\subset(\C^n\!\times S,(\bo,s))$ and $p$ the second projection, is a deformation of $(X,\bo)$ over $(S,s)$.
\end{proposition}

\begin{proof}
Since $f_1,\dots,f_k$ is a regular sequence, any relation among
$f_1, \dots, f_k$ can be generated by the {\em trivial
  relations}\index{trivial!relations} (also called the {\em Koszul
  relations}\index{Koszul!relations}) 
$$ (0,\dots,0,-f_j,0,\dots,0,f_i,0\dots,0) $$
with $-f_j$ at place $i$ and $f_i$ at place $j$. This can be easily
shown by induction on $k$. Another way to see this is to use the
Koszul complex of \mbox{$\boldsymbol{f}=(f_1,\dots,f_k)$}: we  
have
$$ H_1(f,\ko_{\C^n,\bo})=\{\text{relations between }
f_1,\dots,f_k\}/\{\text{trivial relations}\}\,, $$
and $H_1(f,\ko_{\C^n,\bo})=0$ if
$f_1,\dots,f_k$ is a regular sequence (\cite[Theorem B.6.3]{GLS07}). 
Since the trivial relations can obviously be lifted, the result follows
from Proposition \ref{prop:flatbyrel}.
\end{proof}

Let us finish this section with two concrete examples. 

\begin{example}\label{ex.sing1}
{\em (1) Let $(X,\bo)\subset(\C^3\!,\bo)$ be the curve singularity given by
$f_1=xy$, $f_2=xz$, $f_3=yz$. Consider the unfolding of $(f_1,f_2,f_3)$ 
over $(\C,0)$ given by $F_1=xy-t\,,\ F_2=xz\,,\ F_3=yz. $
It is not difficult to check that the sequence 
$$
 0 \longleftarrow \ko_{X,\bo} \longleftarrow  \ko_{\C^3\!,\bo} \xleftarrow{(xy,xz,yz)} 
 \ko_{\C^3\!,\bo}^3 \xleftarrow{\Big(\begin{smallmatrix}0 & -z\\ -y & y\\ x &
 0 \end{smallmatrix}\Big)}\ko_{\C^3\!,\bo}^2 \longleftarrow 0\,, $$
is exact and, hence, a free resolution
of $\ko_{X,\bo}=\ko_{\C^3\!,\bo}/\langle  f_1,f_2,f_3\rangle$. 
That is, \mbox{$(0,-y,x)$} and \mbox{$(-z,y,0)$} generate the
\mbox{$\ko_{\C^3\!,\bo}$}-module of relations between $xy,xz,yz$.

Similarly, we find that \mbox{$(0,-y,x)$}, \mbox{$(yz,-y^2,t)$},
\mbox{$(xz,t-xy,0)$} generate the \mbox{$\ko_{\C^3\!,\bo}$}-module of relations
of $F_1,F_2,F_3$. The liftable relations for $f_1,f_2,f_3$ are obtained from
these by setting \mbox{$t=0$}, which
shows that the relation \mbox{$(-z,y,0)$} cannot be lifted. Hence,
\mbox{$\ko_{\C^3\times\C,\bo}/\langle F_1,F_2,F_3\rangle$} is not
$\ko_{\C,0}$-flat and, therefore, the above unfolding does not define a
deformation of $(X,\bo)$.
We check this in the following {\sc Singular} session:

\begin{small}
\begin{verbatim}
ring R = 0,(x,y,z,t),ds;
ideal f = xy,xz,yz;
ideal F = xy-t,xz,yz;
module Sf = syz(f);      // the module of relations of f
print(Sf);               // shows the matrix of Sf
//-> 0, -z,
//-> -y,y, 
//-> x, 0 
syz(Sf);                 // is 0 iff the matrix of Sf injective
//-> _[1]=0
module SF = syz(F);
print(SF);
//-> 0, yz, xz,  
//-> -y,-y2,t-xy,
//-> x, t,  0    
\end{verbatim}
\end{small}

\noindent
To show that the relation $(-z,y,0)$ in \texttt{Sf} cannot be lifted to
\texttt{SF}, we substitute \texttt{t} by zero in \texttt{SF} and show that
\texttt{Sf} is not contained in the module obtained 
(\texttt{reduce(Sf,std(subst(SF,t,0)))} does produce zero in  {\sc Singular}). 


\medskip\noindent
(2) However, if we consider the unfolding
$ F_1=xy-tx\,,\ F_2=xz\,,\ F_3=yz $ of $(f_1,f_2,f_3)$, we obtain $(-z,-t,x)$,
\mbox{$(-z,y-t,0)$} as generators of the relations 
among $F_1,F_2,F_3$. 

Since \mbox{$(0,-y,x)=(-z,0,x)-(-z,y,0)$}, it
follows that any relation among $f_1,f_2,f_3$ can be lifted. Hence,
\mbox{$\ko_{\C^3\times\C,\bo}/\langle F_1,F_2,F_3\rangle$} is
$\ko_{\C,\bo}$-flat and the diagram
$$\xymatrix@C=16pt@R=12pt@M=5pt{
(X,\bo)\ \ar@{^{(}->}[r]\ar[d] & V(F_1,F_2,F_3) \ar[d] &
\hspace*{-20pt}\subset(\C^3\!\times\C,(\bo,0))\\
\{ 0\}\ \ar@{^{(}->}[r] & (\C,0)}. $$
defines a deformation of $(X,0)$.}
\end{example}

\bigskip 

\subsection{Versal Deformations}
\setcounter{equation}{0}

\noindent
A versal deformation of a complex space germ is a deformation which
contains basically all information about any possible deformation of
this germ. A semiuniversal deformation is minimal versal deformation. 
It is one of the fundamental facts of singularity theory
that any isolated singularity $(X,x)$ has a semiuniversal deformation. 

In a little less informal way we say that a deformation $(i,\phi)$ of
$(X,x)$ over $(S,s)$ is {\em versal} if any other deformation of $(X,x)$
over some base space $(T,t)$ can be induced from $(i,\phi)$ by some
base change $\varphi\colon(T,t)\to(S,s)$. Moreover, if a
deformation of $(X,x)$ over some subgerm $(T'\!,t)\subset(T,t)$
is given and induced by some base change
$\varphi'\colon(T'\!,t)\to(S,s)$, then $\varphi$ can be chosen in
such a way that it extends $\varphi'$. This fact is important, though
it seems a bit technical, as it allows us to construct versal
deformations by successively extending over bigger and bigger spaces
in a formal manner (see \cite{GLS07}, Appendix C for
general fundamental facts about formal deformations, in particular, 
Theorem C.1.6, and the sketch of its proof).  
\medskip

The formal definition of a (semiuni-) versal deformation is as follows.
\begin{definition} \label{def:versal}
\begin{enumerate}[leftmargin=*]
\itemsep3pt
\item[(1)] A deformation
$(X,x)\stackrel{i}{\hookrightarrow}(\sX,x)
  \stackrel{\phi}{\to}(S,s)$ of $(X,x)$ is called {\em
    complete}\index{complete!deformation}\index{deformation!complete} 
  if, for any deformation
$(j,\psi)\colon (X,x)\hookrightarrow(\sY,y)\to(T,t)$
  of $(X,x)$, there exists a morphism
$\varphi\colon(T,t)\to(S,s)$ such that $(j,\psi)$ is
  isomorphic to the induced deformation
$(\varphi^{\ast}i,\varphi^{\ast}\phi)$. 
\item[(2)] The deformation $(i,\phi)$ is called {\em
    versal\/}\index{versal!deformation}\index{deformation!versal} (respectively
{\em formally versal}\index{formally!versal}\index{deformation!formally
  versal}\index{versal!formally}) 
if, for a given deformation $(j,\psi)$ as above the following holds:
for any closed embedding $k\colon(T'\!,t)\hookrightarrow(T,t)$ of 
complex germs (respectively of Artinian complex germs\,\footnote{A complex germ consisting of one point with local ring an Artinian local ring. It is also called
a {\em fat point}.})\index{Artinian complex germ}\index{complex germ!Artinian}
\index{fat point}
and any morphism $\varphi'\colon(T'\!,t)\to(S,s)$ such that
$(\varphi'\,^{\ast}i,\varphi'\,^{\ast}\phi)$ is isomorphic to
$(k^{\ast}j,k^{\ast}\psi)$ there exists a morphism
$\varphi\colon(T,t)\to(S,s)$ satisfying
\begin{enumerate}
\item[(i)] \ $\varphi\circ k=\varphi'$, and
\item[(ii)] \ $(j,\psi)\cong(\varphi^{\ast}i,\varphi^{\ast}\phi)$.
\end{enumerate}
That is, there exists a commutative diagram with Cartesian squares
$$\xymatrix@C=25pt@R=18pt@M=6pt
{
  & (X,x)\ar@{_{(}->}[dl]_{k^{\ast}j}\ar@{^{(}->}[d]^j
  \ar@{^{(}->}[dr]^i \\  
  k^{\ast}(\sY,y) \ar@{}[dr]|{\Box}\ar@{^{(}->}[r]\ar[d]_{k^{\ast}\psi} & 
  (\sY,y)\ar[d]_{\psi} \ar@{-->}[r]\ar@{}[dr]|{\Box} &  (\sX,x)\ar[d]^{\phi}\\
  (T'\!,t) \ar@/_1pc/@<-2ex>[rr]_{\varphi'}\ar@{^{(}->}[r]_k  & (T,t)
  \ar@{-->}[r]_{\varphi} &  (S,s)\ .} $$
\item[{(3)}] A (formally) versal deformation is called {\em
    semiuniversal} if, with the notations of (2), the 
  Zariski tangent map $ T_{(T,t)}\to T_{(S,s)} $ of $\varphi$
  is uniquely determined by $(i,\phi)$ and $(j,\psi)$.
\end{enumerate}
\end{definition}

\noindent
A semiuniversal deformation is also called
 {\em miniversal}\index{deformation!semiuniversal}\index{deformation!miniversal}\index{semiuniversal!deformation}\index{miniversal deformation} because the Zariski tangent space of its base space has the smallest possible dimension among all versal deformations.
 Note that we do not consider {\em universal} deformations (i.e., $\varphi$ in (3) itself is uniquely determined) as this would be too restrictive. 

A versal deformation is
complete (take as \mbox{$(T'\!,t)$} the reduced point $\{ s\}$), but the
converse is not true in general. In the literature the distinction
between complete and versal deformations is not always sharp, some
authors call complete deformations (in our sense) versal. However, the
full strength of versal (and, hence, semiuniversal) deformations comes from
the property requested in (2).

If we know a versal deformation of $(X,x)$, we know, at least in
principle, all other deformations (up to the knowledge of the base
change map $\varphi$). In particular, we know all nearby fibres and,
hence, all nearby singularities which can appear for an arbitrary deformation
of $(X,x)$. 

An arbitrary complex space germ may not have a versal deformation. It
is a fundamental theorem of Grauert \cite{Gr72} that for isolated
singularities a semi\-uni\-versal deformation exists. 
\medskip

Recall that $(X,x)$
has an {\em isolated singularity}\index{singularity!isolated}\index{isolated singularity}, if there exists a representative $X$ with 
$X \setminus \{x\}$ nonsingular. A point $y$ of $X$ is called {\em nonsingular}
\index{nonsingular} or {\em smooth}\index{smooth} 
if $X$ is a complex manifold in a neighbourhood of $y$ 
(equivalently, the local ring $\ko_{X,y}$ is not regular),  otherwise $y$ is called a {\em singular point} \index{singular point} of $X$.

\begin{theorem}[Grauert, 1972]\index{Grauert}
\label{thm:grauert}
Any complex space germ $(X,x)$ with isolated singularity\,\footnote{More
  generally, a semiuniversal deformation exists if
  \mbox{$\dim_{\C}T^1_{(X,x)}<\infty$} (see Definition \ref{def:inf def:1}).} has a semiuniversal
deformation\index{semiuniversal!deformation} 
$$
(X,x)\stackrel{i}{\hookrightarrow}(\sX,x)\stackrel{\phi}{\to}(S,s)\,.
$$
\end{theorem}

\noindent
In Theorem \ref{thm:exis vers def} we describe the semiuniversal deformation explicitly if $(X,x)$ is an isolated complete intersection. For the 
procedure to construct a formal semiuniversal deformation in general by induction, see the beginning of Section \ref{sec:obs}. For an equivariant weighted homogeneous version see  (the proof of) Theorem \ref{thm:pinkham}.

\medskip
Even knowing that a semiuniversal deformation of an
isolated singularity $(X,x)$ exists,  in general we cannot say anything in advance about its
structure. For instance, we can say nothing
about the dimension of the base space of the semiuniversal
deformation, which we shortly call the {\em semiuniversal base
space}\index{semiuniversal!base space}. It is unknown (but believed),
if any complex space germ can occur as a semiuniversal base of an
isolated singularity. Further questions are wether $(X,x)$ is 
{\em smoothable}\index{smoothable}, i.e., if there are nearby fibres 
that are smooth, or if $(X,x)$ is {\em rigid}, i.e., if it cannot be deformed at all
(cf. Section  \ref{sec:rigsmo} for details).

The following Lemma is an easy consequence of the inverse function theorem.
\begin{lemma}
If a semiuniversal deformation of a complex space germ $(X,x)$ exists,
then it is uniquely determined up to (non unique) isomorphism.
\end{lemma}

%

We mention some properties of versal deformations, which  hold in a much
more general deformation theoretic context (see Remark \cite[C.1.5.1]{GLS07}).

\begin{theorem}\label{thm:flenner versal}
If a versal deformation of $(X,x)$ exists then there exists
also a semiuniversal deformation, and every deformation of $(X,x)$ which is 
formally versal is also versal. 
\end{theorem} 

For the proof see \cite[Satz 5.2]{Fl81}. It is based on the following useful
result (c.f. \cite[Proposition I.1.14]{GLS07}): 

\begin{proposition}\label{prop:versal -- semiuniversal}
Every versal deformation of $(X,x)$ differs from the semiuniversal
deformation by a smooth factor.

More precisely, let
\mbox{$\phi:({\sX},x)\to(S,s)$}
be the semiuniversal deformation and
\mbox{$\psi:({\sY},y)\to(T,t)$}
a versal deformation of $(X,x)$. Then there exists a \mbox{$p\geq 0$} and an
isomorphism 
$$
\varphi\colon(T,t)\stackrel{\cong}{\longrightarrow}(S,s)\times (\C^p\!,\bo)
$$
such that \mbox{$\psi\cong(\pi\circ\varphi)^{\ast}\phi$} where
\mbox{$\pi\colon (S,s)\times(\C^p\!,\bo)\to(S,s)$} is the projection on
the first factor.  
\end{proposition}

\begin{remark}
(1) A formula for the extra smooth factor in Propositoin \ref{prop:versal -- semiuniversal} is given in Corollary  \ref{cor:extrafactor}.

(2) The statements of \ref{thm:grauert}\,--\,\ref{prop:versal -- semiuniversal}  
hold also for 
{\em multigerms\/}\index{multigerm}
$(X,x)=\coprod_{\ell=1}^r (X_\ell,x_\ell)\,,$
that is, for the disjoint union of finitely many germs (the existence as in Theorem
\ref{thm:grauert} is assured if all germs $(X_\ell,x_\ell)$ have
isolated singular points). A {\em versal}, resp.\
 {\em semiuniversal,}\index{semiuniversal!deformation!of multigerm}
deformation of the multigerm \mbox{$(X,x)$} is a multigerm
\mbox{$(\bi,\bphi)=\coprod_{\ell=1}^r (i_\ell,\phi_\ell)$} such that,
for each \mbox{$\ell=1,\dots,r$},
$(i_\ell,\phi_\ell)$ is a versal, resp.\ semiuniversal, deformation of
$(X_\ell,x_\ell)$ over  $(S_\ell,s_\ell)$, and the base space of $(\bi,\bphi)$ is the cartesian product of base spaces $(S_\ell,s_\ell)$.
\end{remark}

A semiuniversal deformation is also called miniversal, since it has the
minimal dimension among all versal deformations 
(by Proposition \ref{prop:versal -- semiuniversal}). It has  also another
minimality property, due to Teissier \cite[Theorem 4.8.4]{Te78}.

\begin{theorem}[Economy of the semiuniversal deformation]\label{economy}
\index{semiuniversal!deformation!economy of}
Let $\phi:(\sX,x)\to (S,s)$ be the semiuniversal deformation of
 an isolated singularity $(X,x$). Then, for  any $y \neq x$ sufficiently close to $x$ no fibre $(\sX_{\phi(y)},y)$ is isomorphic to $(X,x)$.
\end{theorem}

This theorem can easily be deduced from the following general result about the trivial locus of a morphism
due to Hauser and M\"uller \cite{HM89}, with special cases proved before in 
\cite[Lemma 1.4]{GKa89} and \cite[Corollary 0.2]{FK87}. Recall that a morphism 
$f:(X,x) \to (S,s)$ of complex germs is called {\em trivial}  if
$(X,x) \cong (f^{-1}(s),x) \times (S,s)$ over $(S,s)$. $f$ is called {\em smooth}
if it is trivial with $(f^{-1}(s),x)$ smooth.
Let $Z_{red}$ denote the reduction of the complex space $Z$.

\begin{theorem}\label{trivial locus}
For any morphism  $f:(X,x) \to (S,s)$ of complex germs there exist complex germs $(Y,x) \subset (X,x)$ and $(T,s)\subset (S,s)$ with the following property
for sufficiently small representatives.
\begin{enumerate}[leftmargin=*]
\item 
$Y_{red} = \{y \in X \,| \,(X,y) \cong (X,x)\}$ and $T_{red}=f(Y).$
\item The restriction $f_Y: Y\to T$ is a smooth morphism.
\item $f^{-1}(s) \cong f_Y^{-1}(s) \times Z$ for some complex space $Z$.
\item If $\varphi: S' \to S$ is a morphism (of germs), then $\varphi^*(f): X \times_{S} S' \to S'$
is trivial iff $\varphi$ factors through $T$.
\end{enumerate}
\end{theorem}
\noindent
The universal property (4) implies that $(T,s)$ is uniquely determined,
while $(Y,x)$ is only determined up to isomorphism over $(T,s)$.
\medskip

For the proof of the following theorem, we refer to \cite{Fl78,Fl81}.

\begin{theorem}[Openness of versality]\index{openness of
    versality}\label{thm:openness of versality} 
Let \mbox{$f:X\to S$} be a flat morphism of complex spaces such that
$\Sing(f)$ is finite over $S$. Then the set of points \mbox{$s\in S$}
such that $f$ induces a versal deformation of the multigerm
\mbox{$\bigl(X,\Sing(f^{-1}(s))\bigr)$} is analytically open in $S$. 
\end{theorem}

Hence, if {$\phi:(\sX,x)\to (S,s)$ is 
a versal deformation of $\bigl(\phi^{-1}(s),x\bigr)$
then, for a sufficiently small representative \mbox{$\phi:\sX\to S$},
the multigerm \mbox{$\phi: \coprod_{x'\in \phi^{-1}(t)} (\sX,x')\to
    (S,t)$}, \mbox{$t\in S$}, 
is a versal deformation of its fibre, the multigerm $\coprod_{x'\in \phi^{-1}(t)}
\bigl(\phi^{-1}(t),x'\bigr)$. The nearby fibre have only isolated singularities, 
since $\Sing(f)\cap  f^{-1}(s)=\Sing(f^{-1}(s))$ is a finite set by assumption.
Note that an analogous statement does not hold for ``semiuniversal''
in place of ``versal''. \\

Although we cannot say anything specific about the semiuniversal
deformation of an arbitrary singularity, the situation is different
for special classes of singularities. For example, hypersurface
singularities or, more generally, complete
intersection\index{complete!intersection} singularities 
are never rigid and we can compute explicitly the semiuniversal
deformation as in the following theorem (for a proof see \cite{KS72} or
 \cite[Theorem I.1.16]{GLS07}).

\begin{theorem}\label{thm:exis vers def}
Let \mbox{$(X,\bo)\subset(\C^n\!,\bo)$} be an isolated complete
intersection\index{complete!intersection} singularity, and let
\mbox{$f:=(f_1,\dots,f_k)$} be a minimal set of generators for the 
ideal of $(X,\bo)$.
Let \mbox{$g_1,\dots,g_{\tau}\in\ko_{\C^n,\bo}^k$},
\mbox{$g_i=(g_i^1,\dots,g_i^k)$},  represent a basis
(respectively a system of generators) for the
finite dimensional $\C$-vector space\,\footnote{The vector space
  \mbox{$T^1_{(X,x)}$} will be defined for arbitrary complex space germs
  $(X,x)$ in Definition \ref{def:inf def:1}. For a definition of $T^1$ in a
  general deformation theoretic context see \cite[Appendix C]{GLS07}.}  
$$ T^1_{(X,\bo)}\index{$t1xx$@$T^1_{(X,x)}$}
:=\ko_{\C^n,\bo}^k\big/\bigl(Df\cdot\ko_{\C^n,\bo}^n+\langle
f_1,\dots,f_k\rangle\ko_{\C^n,\bo}^k\bigr)\,, $$
and set \mbox{$F=(F_1,\dots,F_k)$}, 
\begin{eqnarray*}
F_1(\bx,\bt) & = & f_1(\bx)+\sum_{j=1}^{\tau} t_jg_j^1(\bx)\,,\\[-0.6em]
\vdots \quad & & \qquad \vdots\\
F_k(\bx,\bt) & = & f_k(\bx)+\sum_{j=1}^{\tau} t_jg_j^k(\bx)\,,\\
(\sX,\bo) & := & V(F_1,\dots,F_k)\subset(\C^n\!\times\C^{\tau},\bo)\,. 
\end{eqnarray*}
Then \mbox{$(X,\bo)\stackrel{i}{\hookrightarrow}(\sX,\bo)
\stackrel{\phi}{\to}(\C^{\tau},\bo)$}
with $i,\phi$ being induced by the inclusion
\mbox{$(\C^n\!,\bo)\subset(\C^n\!\times\C^{\tau},\bo)$},
resp. the projection
\mbox{$(\C^n\!\times\C^{\tau},\bo)\to(\C^{\tau},\bo)$}, 
 is a semiuniversal (respectively versal) deformation of
$(X,\bo)$. 
\end{theorem}

\noindent
Here, $Df$ denotes the Jacobian matrix of $f$,
$$ (Df)=\Big(\frac{\partial f_i}{\partial
x_j}\Big)\colon\ko_{\C^n,\bo}^n\longrightarrow\ko_{\C^n,\bo}^k\,, $$
that is, \mbox{$(Df)\cdot\ko_{\C^n,\bo}^n$} is the submodule of
$\ko_{\C^n,\bo}^k$ 
spanned by the columns of the Jacobian matrix of $f$. 

Note that $T^1_{(X,\bo)}$ is an $\ko_{X,\bo}$-module, called the {\em Tjurina
  module\/}\index{Tjurina!module!of complete intersection} of the
complete intersection $(X,\bo)$. If $(X,\bo)$ is a hypersurface, then
$T^1_{(X,\bo)}$ is an algebra and called the {\em Tjurina algebra
  of\/}\index{Tjurina!algebra} $(X,\bo)$.   The number
    $$ \tau(X,x) := \dim_\C T^1_{(X,\bo)}$$
  is called the {\em Tjurina number} of $(X,x)$.\index{Tjurina!number}

\medskip
Since the hypersurface case is of special importance we state it
explicitly.
\begin{corollary}\label{cor:hypersurface case T1}
  Let \mbox{$(X,\bo)\subset(\C^n\!,\bo)$} be an isolated singularity
  defined by \mbox{$f\in\ko_{\C^n,\bo}$} and
  \mbox{$g_1,\dots,g_{\tau}\in\ko_{\C^n,\bo}$} a $\C$-basis of the
  Tjurina algebra
$$ T^1_{(X,\bo)}=\ko_{\C^n,\bo}/\textstyle\big\langle f,\frac{\partial 
f}{\partial x_1},\dots,\frac{\partial f}{\partial
x_n}\big\rangle\,. $$
If we set
$$
F(\bx,\bt) := f(\bx)+\sum_{j=1}^{\tau}t_jg_j(\bx)\,,\quad
(\sX,\bo) := V(F)\subset (\C^n\!\times\C^{\tau},\bo)\,,
$$
then
\mbox{$(X,\bo)\hookrightarrow(\sX,\bo)\xrightarrow{\phi}(\C^{\tau}\!,\bo)$},
with $\phi$ the second projection, is a semiuniversal deformation of
$(X,\bo)$. 
\end{corollary}

\begin{remark}
\label{rmk:vers def}
Using the notation of Theorem \ref{thm:exis vers def}, we can choose the 
basis $g_1,\dots,g_{\tau}\in\ko_{\C^n,0}^k$ of $T^1_{(X,\bo)}$ such that 
$g_i=-e_i$, \mbox{$e_i=(0,\dots,1,\dots,0)$} the $i$-th canonical
generator of $\ko_{\C^n,\bo}^k$, for $i=1,\dots,k$ (assuming that
\mbox{$f_i\in\fm_{\C^n,\bo}^2$}). Then
$$ F_i=f_i-t_i+\sum_{j=k+1}^{\tau}t_jg_j^i\,, $$
and we can eliminate $t_1,\dots,t_k$ from \mbox{$F_1=\ldots
=F_k=0$}. Hence, the semiuniversal deformation of $(X,\bo)$ is given by
$$
\psi: (\C^n\!\times\C^{\tau-k},\bo)
\to(\C^k\times \C^{\tau-k},\bo)=(\C^{\tau}\!,\bo)
$$
with \mbox{$\psi(\bx,t_{k+1},\dots,t_{\tau})  = 
(G_1(\bx,\bt),\dots,G_k(\bx,\bt),t_{k+1},\dots,t_{\tau})$}, 
$$
G_i(\bx,\bt)  =  f_i(\bx)+\sum_{j=k+1}^{\tau}t_jg_j(\bx)\,,
$$
where \mbox{$g_j=(g_j^1,\dots,g_j^k)$}, \mbox{$j=k+1,\dots,\tau$},
is a basis of the $\C$-vector space
$$\bigl(\fm\cdot
\ko_{\C^n\!,\bo}^k\bigr)\big/\bigl((Df)\cdot\ko_{\C^n\!,\bo}^n+\langle
f_1,\dots,f_k\rangle\ko_{\C^n\!,\bo}^k\bigr)\,,$$ 
assuming \mbox{$f_1,\dots,f_k\in\fm_{\C^n\!,\bo}^2$}.

In particular, if \mbox{$f\in\fm_{\C^n\!,\bo}^2$} and if
$1,h_1,\dots,h_{\tau-1}$ is a basis of the Tjurina 
algebra $T_{f}$, then (setting \mbox{$\bt:=(t_1,\dots t_{\tau-1}$})
$$ F\colon(\C^n\!\times\C^{\tau-1}\!,\bo)\longrightarrow(\C^{\tau},\bo)\,,\
\textstyle(\bx,\bt)\mapsto
\Bigl(f(\bx)+\sum\limits_{i=1}^{\tau-1}t_ih_i,\,\bt\Bigr) $$
is a semiuniversal deformation of the hypersurface singularity
\mbox{$(f^{-1}(0),\bo)$}.
\end{remark}

\medskip

\begin{example}
{\em (1)\: Let \mbox{$(X,\bo)\subset (\C^3\!,\bo)$} be the isolated complete
intersection curve singularity 
defined by the vanishing of 
$f_1(\bx)=x_1^2+x_2^3$ and of
$f_2(\bx)=x_3^2+x_2^3$. Then the Tjurina module is
$ T^1_{(X,\bo)}= \C\{\bx\}^2/M$, where $M\subset
      \C\{\bx\}^2$ is generated by
$\tbinom{x_1}{0},\tbinom{x_2^2}{x_2^2},
    \tbinom{0}{x_3},\tbinom{f_1}{0},\tbinom{0}{f_1},\tbinom{f_2}{0},
    \tbinom{0}{f_2}$. 
    We have $\tau=9$ and a $\C$-basis for $T^1_{(X,\bo)}$ is given by
$\tbinom{1}{0},$ $\tbinom{0}{1},$  $\tbinom{x_2}{0},$ 
    $\tbinom{x_3}{0},$ $\tbinom{x_2x_3}{0},$ $\tbinom{x_2^2}{0},$ 
    $\tbinom{0}{x_1},$  $\tbinom{0}{x_2},$  $ \tbinom{0}{x_1x_2}$. 
Again by Remark \ref{rmk:vers def},
    it follows that a semiuniversal deformation of $(X,\bo)$ is given by
    \mbox{$\psi \colon (\C^{10}\!,\bo)\to(\C^9\!,\bo)$}, 
$$
(\bx,\bt)\longmapsto\bigl(f_1(\bx)+t_1x_2+t_2x_3+t_3x_2x_3+t_4x_2^2,\,
f_2(\bx)+t_5x_1+t_6x_2+t_7x_1x_2,\bt\bigr)\,.
$$
This can easiiy verified by a computation in {\sc Singular}.}
\end{example}

\bigskip

\subsection{Infinitesimal Deformations}\index{infinitesimal!deformation}
\setcounter{equation}{0}
\label{sec:inf def and obs}

In this section we develop infinitesimal deformation theory for
{\em arbitrary singularities}. In particular, we introduce in this generality
the vector spaces 
$T^1_{(X,x)}$ of {\em first order
  deformations}
  that is, the linearization of 
the deformations of $(X,x)$ and show how it can be computed. In the next section
we describe the obstructions for lifting an infinitesimal deformation
of a given order to higher order. This and the next section can be
considered as a concrete special case of the general theory 
described in \cite[Appendix C]{GLS07}.

\medskip
Infinitesimal deformation theory of first order is the deformation over
the complex space $T_{\eps}$, a ``point with one tangent direction''.

\begin{definition} \label{def:inf def:1}
\begin{enumerate}[leftmargin=*]
 \item[(1)] The complex space germ
  $T_{\eps}=(${\em\{pt\}},$\C[\eps])$\index{$t$@$\Teps$} consists of one point
  with local ring $\C[\eps]=\C[t]/\langle t^2\rangle$. 
  \item[(2)] For any complex space germ $(X,x)$ we set
       $$ T^1_{(X,x)}\index{$t1xx$@$T^1_{(X,x)}$} :=\uDef_{(X,x)}(T_{\eps})\,,$$
  the set of isomorphism
  classes of deformations of $(X,x)$ over $T_{\eps}$. Objects of 
\mbox{$\Def_{(X,x)}(T_{\eps})$} are called 
 {\em infinitesimal} or  {\em  first order deformations} of $(X,x)$.
\index{first order deformation}\index{deformation!first
  order}\index{deformation!infinitesimal}\index{infinitesimal!deformation!first order}  

\item[(3)] We shall see in Proposition \ref{prop:t1xx}  that $T^1_{(X,x)}$ 
  carries the structure of a complex vector space. We call $T^1_{(X,x)}$ the {\em Tjurina module}\index{Tjurina!module}, and
$$ \tau(X,x) := \dim_{\C}T^1_{(X,x)}$$
 the {\em Tjurina  number}\index{Tjurina!number}\index{$tau$@$\tau(X,x)$} of $(X,x)$.
  \end{enumerate}
\end{definition}

Any singularity $(X,x)$ with \mbox{$\tau(X,x)<\infty$} has
a semiuniversal deformation (see \cite{Gr72,St03}); it is
not difficult to see that isolated singularities have finite
Tjurina number. 
\medskip

The following lemma shows that $T^1_{(X,x)}$ can be identified with
the Zariski tangent space to the semiuniversal base of $(X,x)$ (if it
exists).

\begin{lemma}\label{lem:Kodaira-Spencer map}
Let $(X,x)$ be a complex space germ and $\phi:(\sX,x)\to(S,s)$ 
a deformation of $(X,x)$. Then 
there exists a linear map\,\footnote{$T_{S,s}$ denotes the Zariski
  tangent space to $S$ at $s$.} 
$$ T_{S,s} \lra T^1_{(X,x)}\,, $$
called the {\em Kodaira-Spencer map}\index{Kodaira-Spencer map}, which 
is surjective if $\phi$ is versal and bijective if $\phi$ is
semiuniversal. 

Moreover, if $(X,x)$ admits a semiuniversal deformation
with smooth base space, then $\phi$ is semiuniversal iff $(S,s)$ is smooth and
the Kodaira-Spencer map is an isomorphism.
\end{lemma}
\begin{proof}
For any complex space germ $(S,s)$ we have
$T_{S,s}=\Mor\big(T_{\eps},(S,s)\big)$. Define a map
\begin{eqnarray*}
  \alp\colon\Mor\big(T_{\eps},(S,s)\big) & \longrightarrow &
  T^1_{(X,x)}\,,\\
  \varphi & \mapsto & \big[\varphi^{\ast}\phi\big]\,.
\end{eqnarray*}
Let us see that $\alp$ is surjective if $\phi$ is versal: given a class
\mbox{$[\psi]\in T^1_{(X,x)}$} represented by
\mbox{$\psi:(\sY,x)\rightarrow T_{\eps}$}, the
versality of $\phi$ implies the existence 
 of a map \mbox{$\varphi\colon T_{\eps}\to(S,s)$} such that
\mbox{$\varphi^{\ast}\phi\cong \psi$}. Hence,
\mbox{$[\psi]=\alp(\varphi)$}, and $\alp$ is surjective.

If $\phi$ is semiuniversal, the tangent map $T\varphi$ of 
$\varphi\colon T_{\eps}\to(S,s)$ is uniquely determined by
$\psi$. 
Since $\varphi$ is uniquely determined by its algebra map 
$ \varphi^{\sharp}\colon\ko_{S,s}\to\ko_{T_{\eps}}=\C[t]/\langle
t^2\rangle $
and, since $\varphi^{\sharp}$ is local, we obtain
\mbox{$\varphi^{\sharp}(\fm_{S,s}^2)=0$}. That is, $\varphi$ is
uniquely determined by
$\underline{\varphi}^{\sharp}\colon\fm_{S,s}/\fm_{S,s}^2
\longrightarrow\langle t\rangle/\langle t^2\rangle $
and hence by the dual map
$\big(\underline{\varphi}^{\sharp}\big)^{\ast}=T\varphi$. Thus,
$\alp$ is bijective. We leave the linearity of $\alpha$ as an exercise.

If $(T,t)$ is the smooth base space of a semiuniversal deformation of
$(X,x)$ then there is a morphism \mbox{$\varphi:(S,s)\to (T,t)$}
inducing the map 
$\alpha:T_{S,s}\to T_{T,t}\cong T^1_{(X,x)}$
constructed above. 
Since $(S,s)$ is smooth, $\varphi$ is an isomorphism iff $\alpha$ is (by the
inverse function theorem).
\end{proof}

\begin{remark}\label{ex.sudef}
Lemma \ref{lem:Kodaira-Spencer map} shows that $\dim_\C T^1_{(X,x)}<\infty$ if $(X,x)$ admits a semiuniversal deformation. Together with Theorem \ref{thm:grauert} this shows that $\dim_\C T^1_{(X,x)}<\infty$ is necessary and sufficient  for the 
existence of a semiuniversal deformation of $(X,x)$. If $(X,x)$ has 
an isolated singularity, then $\dim_\C T^1_{(X,x)}<\infty$ by Corollary \ref{T1.iso}
but the converse does not hold (see Example \ref{exa:semiuni def}, below).
\end{remark}

We want to describe now $T^1_{(X,x)}$ in terms of the defining
ideal of $(X,x)$ if  $(X,x)$ is embedded in some $(\C^n\!,\bo)$, 
without knowing a semiuniversal deformation of $(X,x)$. 
To do this, we need embedded deformations, that is, deformations of the inclusion map
\mbox{$(X,x)\hookrightarrow 
  (\C^n\!,\bo)$}. 
  
  Slightly more general, we define deformations of  
a morphism, not necessarily an embedding.
\begin{definition}
\label{def:var defs}
  Let $f\colon(X,x)\to(S,s)$ be a morphism of complex germs.
  \begin{enumerate}[leftmargin=*]
  \item[(1)] A {\em deformation of $f$}, or a {\em deformation of
      \mbox{$(X,x)\to(S,s)$}}\index{deformation!of maps}\index{deformation}, 
  over a germ $(T,t)$ is a Cartesian diagram
  
  $$\xymatrix@C=16pt@R=15pt@M=2pt
  {
  (X,x) \ar@{}[dr]|{\Box}\ar[d]_f\ar@{^{(}->}[r]^-{i} & (\sX,x) \ar[d]^F
    \ar@/^2pc/@<2ex>[dd]^{\phi}\\
  (S,s) \ar@{}[dr]|{\Box}\ar[d]\ar@{^{(}->}[r]^-{j} & (\sS,s)\ar[d]^p\\
  \{pt\} \ar@{^{(}->}[r] & (T,t)} $$
  such that $i$ and $j$ are closed embeddings, and $p$ and $\phi$ are
  flat (hence deformations of $(X,x)$ and of $(S,s)$ over $(T,t)$, but
  $F$ is not supposed to be flat). We denote such a deformation by
  $(i,j,F,p)$ or just by $(F,p)$.
  
\smallskip\noindent
  A {\em morphism}\index{morphism!of deformations} between two deformations
  $(i,j,F,p)$ and 
  $(i',j',F',p')$ of $f$ is a morphism of diagrams, denoted by $(\psi_1,\psi_2,\varphi)$:
  
  $$\xymatrix@C=25pt@R=10pt@M=2pt
  {
  & (X,x)\ar@{_{(}->}[dl]_{i}\ar@{^{(}->}[dr]^{i'}\ar@{-}[d] \\
  (\sX,x)\ar[rr]^(.37){\psi_1}\ar[dd]_F &\ar[d] & (\sX',x')\ar[dd]^{F'} \\
  & (S,s)\ar@{_{(}->}[dl]_{j}\ar@{^{(}->}[dr]^{j'}\ar@{-}[d] \\
  (\sS,s)\ar[rr]^(.37){\psi_2}\ar[dd]_p &\ar[d] & (\sS',s')\ar[dd]^{p'} \\
  & \{pt\}\ar@{_{(}->}[dl]\ar@{^{(}->}[dr] \\
  (T,t)\ar[rr]^{\varphi} & & (T'\!,t')} $$
  
  \noindent
  If $\psi_1,\psi_2,\varphi$ are isomorphisms, then
  $(\psi_1,\psi_2,\varphi)$ is an isomorphism of deformations of $f$.
  
  \smallskip\noindent 
  We denote by
  \mbox{$\Def\!_f=\Def_{(X,x)\to(S,s)}$}\index{$deff$@$\Def_f$}\index{$defxxtss$@$\Def_{(X,x)\to(S,s)}$}
  the category of deformations of $f$, by
  \mbox{$\Def\!_f(T,t)=\Def_{(X,x)\to(S,s)}(T,t)$} the (non-full)
  subcategory of deformations of $f$ over $(T,t)$ with morphisms as
  above and $\varphi$ the identity on $(T,t)$.
  Furthermore, we write\index{$deff1$@$\uDef_f$}
  $$ \uDef_f(T,t)=\uDef_{(X,x)\to(S,s)}(T,t) $$
  for the set of isomorphism
  classes of such deformations. 
   \smallskip\noindent 
\item[(2)] A deformation $(i,j,F,p)$ of \mbox{$(X,x)\to(S,s)$}
 inducing the trivial deformation of $(S,s)$ 
 is called a {\em deformation of
    $(X,x)/(S,s)$ over $(T,t)$}\index{deformation!of
    $(X,x)/(S,s)$}\index{deformation} 
  and denoted by $(i,F)$ or just by $F$.  A morphism 
  is a morphism as in (1) of the form
  $(\psi,\id_{S,s}\!\!\:\times\varphi,\varphi)$; it is
  denoted by $(\psi,\varphi)$.
  
\noindent
  $\Def_{(X,x)/(S,s)}$\index{$defxxss1$@$\Def_{(X,x)/(S,s)}$} denotes
  the category of deformations of $(X,x)/(S,s)$, 
  $\Def_{(X,x)/(S,s)}(T,t)$ the subcategory of
  deformations of \mbox{$(X,x)/(S,s)$} over $(T,t)$ with morphisms
  being the identity on $(T,t)$. 
  
\noindent
   $\uDef_{(X,x)/(S,s)}(T,t)$ denotes
  the set of isomorphism classes of such deformations.
  \end{enumerate}
\end{definition}

The difference between (1) and (2) is that in (1) we deform $(X,x)$,
$(S,s)$ and $f$, while in (2) we only deform $(X,x)$ and $f$ but not
$(S,s)$. Note that $\Def_{(X,x)/\pt}=\Def_{(X,x)}.$
\medskip

The following easy lemma shows that embedded deformations are a special
case of Definition \ref{def:var defs}\,(2).
\begin{lemma}
  Let \mbox{$f\colon(X,x)\to(S,s)$} be a closed embedding of complex
  space germs and let 
$ (\sX,x)\xrightarrow{F} (\sS,s)\xrightarrow{p}
    (T,t)$
be a deformation of $f$. Then
  \mbox{$F\colon(\sX,x)\to(\sS,s)$} is a closed embedding, too. 
\end{lemma}
\medskip

\begin{definition}
\label{def:embedded defs}
  \begin{enumerate}[leftmargin=*]
  \itemsep3pt
  \item[(1)] Let \mbox{$(X,x)\hookrightarrow(S,s)$} be a closed
  embedding. The objects of \mbox{$\Def_{(X,x)/(S,s)}$} are called
  {\em embedded deformations}\index{embedded!deformation}
  \index{deformation!embedded} of $(X,x)$ (in $(S,s)$).
  \item[(2)] For an arbitrary morphism \mbox{$f\colon(X,x)\to(S,s)$}  we
    define\index{$t1xxss$@$T^1_{(X,x)\to(S,s)}$}\index{$t1xxss1$@$T^1_{(X,x)/(S,s)}$}  
  $$ 
  T^1_{(X,x)\to(S,s)}:=\uDef_{(X,x)\to(S,s)}(T_{\eps})\,, 
  $$
  respectively
  $$
  T^1_{(X,x)/(S,s)}:=\uDef_{(X,x)/(S,s)}(T_{\eps})\,,
  $$
  and call its elements the isomorphism classes of {\em (first order)
    infinitesimal deformations of $(X,x)\to(S,s)$}\index{deformation!of
    $(X,x)\to(S,s)$!infinitesimal}, respectively of
  $(X,x)/(S,s)$\index{deformation!of
    $(X,x)/(S,s)$!infinitesimal}\index{deformation}.  
  \end{enumerate}
\end{definition}

\noindent
The vector space structure on $T^1_{(X,\bo)/(\C^n\!,\bo)}$ and
$T^1_{(X,\bo)}$ is 
given by the isomorphisms in Proposition \ref{prop:t1xx}.  
We are going to describe $T^1_{(X,\bo)/(\C^n\!,\bo)}$ and
$T^1_{(X,0)}$ in terms of the equations defining \mbox{$(X,0)\subset
 (\C^n\!,\bo)$}. 
First, we need some preparations. 
 
\begin{definition}
Let $S$ be a smooth $n$-dimensional complex
manifold and \mbox{$X\subset S$} a complex subspace given by the coherent
ideal sheaf \mbox{$\ki\subset\ko_S$}.
  \begin{enumerate}[leftmargin=*]
\itemsep3pt
  \item[(1)] The sheaf \mbox{$(\ki/\ki^2)\big|_X$} is called the {\em
      conormal sheaf\/}\index{sheaf!conormal}\index{conormal sheaf} and
    its dual\index{$n_xs$@$\kn_{X/S}$} 
  $$ \kn_{X/S}:=
    \kHom_{\ko_X}\bigl((\ki/\ki^2)\big|_X,\ko_X\bigr) $$
  is called the {\em normal
    sheaf\/}\index{sheaf!normal}\index{normal!sheaf} 
  of the embedding \mbox{$X\subset S$}.
\item[(2)] Let
  \mbox{$\Omega^1_X=\bigl(\Omega^1_S/(\ki \cdot \Omega^1_S+d\ki\cdot\ko_S)\bigr)\big|_X$}\index{$omega1X$@$\Omega^1_X$}
  be the sheaf of holomorphic $1$-forms on $X$. 
  The dual sheaf\index{$thetaX$@$\Theta_X$}
  \mbox{$ \Theta_X:=\kHom_{\ko_X}(\Omega^1_X,\ko_X) $}
  is called the sheaf of {\em holomorphic vector fields}
  \index{vector fields}\index{sheaf!of holomorphic vector fields}
  \index{tangent!sheaf} on $X$.  
  \end{enumerate}
\end{definition}

\noindent
Recall that, for each coherent
$\ko_X$-sheaf $\km$, there is a canonical isomorphism of
$\ko_X$-modules 
$$\kHom_{\ko_X}(\Omega^1_X,\km)\stackrel{\cong}{\longrightarrow}
\kDer_{\C}(\ko_X,\km)\,,\quad \varphi\longmapsto\varphi\circ
  d\,,  $$
where \mbox{$d\colon\ko_X\to\Omega^1_X$} is the exterior
derivation and where
$\kDer_{\C}(\ko_X,\km)$\index{$dercom$@$\kDer_{\C}(\ko_X,\km)$} is the
{\em sheaf of $\C$-derivations of $\ko_X$ with values in
  $\km$}\index{sheaf!of derivations}. In particular, we have
$$ \Theta_X\cong\kDer_{\C}(\ko_X,\ko_X)\,. $$
Moreover, recall that the sheaf
$\Omega^1_S$ is 
locally free with
$ \Omega^1_{S,s} = \bigoplus_{i=1}^n\ko_{S,s}dx_i $
(where $x_1,\dots,x_n$ are local coordinates of $S$ with center $s$). As a
consequence we have that $\Theta_S$ is locally free of rank $n$ and
$$ \Theta_{S,s} =
   \bigoplus\limits_{i=1}^n\,
\ko_{S,s}\cdot\frac{\partial}{\partial x_i} $$
where $\frac{\partial}{\partial x_1},\dots,\frac{\partial}{\partial
x_n}$ is the dual basis of $dx_1,\dots,dx_n$.

Let \mbox{$f\in \ko_S$} then, in local coordinates, we have
\mbox{$df=\sum_{i=1}^n\frac{\partial 
f}{\partial x_i}dx_i$}. In particular, we can define an $\ko_S$-linear
map \mbox{$\alpha\colon\ki\to \Omega^1_S$}, \mbox{$f\mapsto df$}. Due
to the Leibniz rule, $\alpha$ induces a map
\mbox{$\alpha:\ki/\ki^2\to \Omega^1_S\otimes_{\ko_S}
\ko_X$} yielding the following exact {\em Zariski-Jacobi}\index{Zariski-Jacobi sequence} sequence
\begin{equation*}
\ki/\ki^2\stackrel{\alp}{\longrightarrow}\Omega^1_S\otimes_{\ko_S}
\ko_X\longrightarrow\Omega^1_X\longrightarrow 0\,.
\end{equation*}
By $\ko_X$-dualizing , we obtain the exact sequence
\begin{equation*}
0\longrightarrow\Theta_X\longrightarrow\Theta_S\otimes_{\ko_S}
\ko_X\stackrel{\bet}{\longrightarrow}\kn_{X/S} = \kHom_{\ko_X} (\ki/\ki^2,\ko_X)\,,
\end{equation*}
where $\beta$ is the dual of $\alpha$. In local coordinates, we have for each
\mbox{$x\in X$} 
$$ \Theta_{S,s}\otimes_{\ko_{S,s}}\ko_{X,x}=\bigoplus_{i=1}^n
\ko_{X,x}\cdot\tfrac{\partial}{\partial x_i} $$
and the image
$ \bet\big(\tfrac{\partial}{\partial x_i}\big)\in
\Hom_{\ko_{X,x}}(\ki_{x}/\ki_{x}^2,\ko_{X,x})=
\Hom_{\ko_{X,x}}(\ki_{x},\ko_{X,x}) $
sends a residue class \mbox{$[h]\in\ki_x/\ki_x^2$} to
\mbox{$\big[\frac{\partial h} {\partial x_i}\big]\in\ko_{X,x}$}.
Using these notations we can describe the vector space structure of
\mbox{$T^1_{(X,\bo)/(\C^n\!,\bo)}$} and of  \mbox{$T^1_{(X,\bo)}$}:

\begin{proposition}\label{prop:t1xx}
  Let \mbox{$(X,\bo)\subset(\C^n\!,\bo)$} be a complex space germ and let
  \mbox{$\ko_{X,\bo}=\ko_{\C^n,\bo}/I$}. Then
  \begin{enumerate}[leftmargin=*]
  \itemsep3pt
  \item \mbox{$T^1_{(X,\bo)/(\C^n\!,\bo)}\cong
  \kn_{X/\C^n\!,\bo}\cong\Hom_{\ko_{\C^n\!,\bo}}(I,\ko_{X,\bo})\,.
  $}\index{$t1xxss1$@$T^1_{(X,x)/(S,s)}$} 
  \item
    \mbox{$T^1_{(X,\bo)}\cong\coker(\bet)$}\index{$t1xx$@$T^1_{(X,x)}$},
    that is, we have an  exact sequence
  $$ 0\longrightarrow\Theta_{X,\bo}\longrightarrow\Theta_{\C^n,\bo}
  \otimes_{\ko_{\C^n,\bo}}\ko_{X,\bo}\stackrel{\bet}{\longrightarrow}
  \kn_{X/\C^n,\bo}\longrightarrow T^1_{(X,\bo)}\longrightarrow 0\,, $$
  where \mbox{$\beta\bigl(\frac{\partial}{\partial x_i}\bigr)\in
    \Hom(I,\ko_{X,\bo})$} sends \mbox{$h\in I$} to the class of
  \mbox{$\frac{\partial h}{\partial x_i}$} in $\ko_{X,\bo}$.
  \item If $(X,\bo)$ is reduced then 
  $T^1_{(X,\bo)}\cong \Ext^1_{\ko_{X,x}}(\Omega^1_{X,x},\ko_{X,x})$.
  \end{enumerate}
\end{proposition}
\noindent
For the proof of (1) and (2) we refer to \cite[Proposition I.1.25]{GLS07} or \cite{St03}. To see $(3)$ note that $\ki/\ki^2$ is free on the regular locus of $X$ and hence $ker(\alpha)$ is concentrated on the singular locus of $X$ and hence torsion since $X$ is reduced. It follows that the dual of
$\ki/\ki^2$ coincides with the dual of $(\ki/\ki^2)/ker(\alpha)$, which implies the claim.
\medskip

\begin{corollary} \label{T1.iso}
$\dim_\C T^1_{(X,x)}<\infty$ if $(X,x)$ is an isolated singularity.
 \end {corollary}
 
\medskip

\begin{remark}\label{rmk:t1xx} The proof of Proposition \ref{prop:t1xx} shows the following:
  \begin{enumerate}[leftmargin=*]
 \item[(1)]   If \mbox{$\ko_{X,\bo}=\ko_{\C^n,\bo}/I$}, \mbox{$I=\langle
      f_1,\dots,f_k\rangle$}, 
  then any embedded deformation of $(X,\bo) \subset (\C^n\!,\bo)$ over $T_{\eps}$ 
  is given by  $f_i+\eps g_i$, $i=1,\dots,k$, $g_i\in \ko_{\C^n,\bo}$, 
which we identify with $(g_1,\dots,g_k)$. We  define a map
  \begin{eqnarray*}
    \gamma\colon T^1_{(X,\bo)/(\C^n\!,\bo)} & \longrightarrow &
    \kn_{X/\C^n\!,\bo}\cong\Hom_{\ko_{\C^n,\bo}}(I,\ko_{X,\bo})\,,\\
    (g_1,\dots,g_k) & \longmapsto & \Bigl(\varphi\colon
    \textstyle\sum\limits_{i=1}^k a_if_i\mapsto \sum\limits_{i=1}^k
    \big[a_ig_i\big]\Bigr)\,,
  \end{eqnarray*}
 which is well-defined, since any relation $\sum_{i=1}^k r_if_i=0$ lifts to a relation 
 $ \sum_{i=1}^k (r_i+\eps s_i)(f_i+\eps g_i)=0$ (by flatness, cf. Proposition \ref{prop:flatbyrel}) and hence $\sum_i r_ig_i\in I$. 
  \item[(2)]  Let $F=(F_1,\dots,F_k)$ be an embedded deformation of $(X,\bo)$ over $T_{\eps}$ given by $F_i=f_i+\eps g_i, \, i=1,\dots,k,$ as in (1)
  such that \mbox{$\sum_i   r_ig_i\in I$} for each relation $(r_1,\dots,r_k)$ among
  $f_1,\dots,f_k$. 
  Then  $F$ and $F'=(F'_1,\dots,F'_k)$, $F'_i=f_i+\eps g'_i$,
  define isomorphic embedded deformations over $T_{\eps}$ iff
  \mbox{$g_i-g'_i\in I$}. The vector space structure on the space 
  of embedded deformations is given by
  \begin{eqnarray*}
    F+F' & = &
    \big(f_1+\eps(g_1+g'_1),\dots,f_k+\eps(g_k+g'_k)\big)\,,\\
    \lambda F & = & (f_1+\eps\lambda g_1,\dots,f_k+\eps\lambda g_k)\,,
    \quad\lambda\in\C\,.
  \end{eqnarray*}
  \item[(3)] The embedded deformation defined by $F$ as above is trivial as
  abstract deformation iff there is a vector field
  \mbox{$\partial=\sum_{j=1}^n\delta_j\frac{\partial}{\partial
    x_j}\in\Theta_{\C^n,\bo}$} such that
  $$ g_i=\partial(f_i)\ \mod I\,, \quad i=1,\dots,k\,. $$
In particular, if \mbox{$I=\langle f\rangle$} defines a hypersurface
singularity, then \mbox{$f+\varepsilon g$} is trivial as abstract deformation
iff $g\in \langle f,\frac{\partial f}{\partial x_j}\mid
  j=1,\dots,n\rangle$. 
  \item[(4]) By (2) and (3) the map $\gamma$ of (1) is an isomorphism. Using
  $\gamma^{-1}$, the morphism $\beta$ from Proposition \ref{prop:t1xx}(2) 
  maps $\frac{\partial}{\partial x_j}$ to 
  $(\frac{\partial f_1}{\partial x_j},...,\frac{\partial f_k}{\partial x_j})$
  since $\beta(\frac{\partial}{\partial x_j})(\sum_{i=1}^k a_if_i) 
   = \sum_{i=1}^k \big[a_i\frac{\partial f_i}{\partial x_j}\big]$. 

  \end{enumerate}
\end{remark}



Proposition \ref{prop:t1xx} provides an algorithm for computing
$T^1_{(X,\bo)}$. This algorithm is implemented in the {\sc
  Singular} library
\texttt{sing.lib}\index{sing.lib@\texttt{sing.lib}}. The {\sc Singular}
procedure 
\texttt{T\_1} computes all relevant information about first order
deformations. For details we refer to \cite[Section I.1.4]{GLS07}. 

\medskip

Infinitesimal deformations are the first step in formal deformation
theory as developed by  Schlessinger\index{Schlessinger} in a very
general context (see \cite[Appendix C]{GLS07} for a short
overview). Schlessinger introduced what is nowadays called  
the {\em Schlessinger conditions}\index{Schlessinger!conditions}
(H$_0$)\,--\,(H$_4$) in \cite{Sc68}.
One can verify that \mbox{$\uDef_{(X,x)}$} satisfies
conditions (H$_0$)\,--\,(H$_3$) and, therefore, has a formal versal
deformation. Moreover, for every deformation functor satisfying the
Schlessinger conditions, the corresponding infinitesimal deformations
carry a natural vector space structure. For $T^1_{(X,x)}$ this
structure coincides with the one defined above.  A
survey of deformations of complex spaces is given in \cite{Pa90}, some aspects
of deformations of singularities are covered by \cite{St03}.
\bigskip

\subsection{Obstructions}\label{sec:obs}
\setcounter{equation}{0}

\noindent
We have seen in Remark \ref{ex.sudef} that  $\dim_\C T^1_{(X,x)}<\infty$ is a necessary and sufficient condition for $(X,x)$ to admit a semiuniversal deformation. However, the existence says nothing about the semiuniversal base space. Some information is contained in
the vector space $T^2_{(X,x)}$, which we describe below. This vector space contains the obstructions to extend a given deformation of $(X,x)$ over a fat point to a bigger one. 
\medskip

\noindent
The construction of a semiuniversal deformation for a complex germ
$(X,x)$ with $\dim_\C T^1_{(X,x)}<\infty$ can be carried out as
follows (for a $\C^*$-equivariant version see the proof of Theorem \ref{thm:pinkham}): 
\begin{itemize}[leftmargin=*]
\item We start with {\em first order deformations\/} and try to {\em lift
    these to 
  second order\/}\index{deformation!second order}\index{second order
  deformation}\index{infinitesimal!deformation!second order}
deformations. In other words, we are looking for possible 
  liftings of a deformation $(i,\phi)$,  \mbox{$[(i,\phi)]\in
    \uDef_{(X,x)}(T_{\eps})= T^1_{(X,x)}$}, to a
  deformation over the fat point point \mbox{$(T'\!,\bo)$} containing $\Teps$,
  for example to the fat point with local ring
  $\C[\eta]/\langle \eta^3\rangle$. Or, if we assume the deformations to be
  embedded (Corollary \ref{cor:defs can be embedded}), this means 
  that we are looking for a lifting of the first order deformation
  \mbox{$f_i+\varepsilon g_i$}, \mbox{$\eps^2=0$}, to a second order
  deformation \mbox{$f_i+\eta g_i+\eta^2 g'_i$}, \mbox{$\eta^3=0$},
  \mbox{$i=1,\dots,k$}. 
\item This is exactly what we did when we constructed the
  semiuniversal deformation of a complete intersection singularity. By 
  induction we showed the existence of a lifting to arbitrarily high
  order. In general, however, this is not always possible, there are
  {\em obstructions\/} against lifting. Indeed, there is an
  $\ko_{X,x}$-module
  $T^2_{(X,x)}$\index{$t2xx$@$T^2_{(X,x)}$}\index{obstruction!module} and, for
  each small extension of $\Teps$, an 
  {\em obstruction map\/}\index{$ob$@$\ob$}\index{obstruction!map}
$$ \ob:  T^1_{(X,x)}  \lra
T^2_{(X,x)} $$
such that the vanishing of \mbox{$\ob\bigl([(i,\phi)]\bigr)$} is equivalent to
the existence of a lifting of $(i,\phi)$ to the small extension, e.g.\ to
second order as above.  
\item Assuming that the obstruction is zero, we choose a lifting to
  second order (which is, in general, not unique) and try to lift this 
  to {\em third order}, that is, to a deformation over the fat
  point with local ring $\C[t]/\langle t^4\rangle$.
 Again, there is an obstruction map, and the
 lifting is possible iff it maps the deformation class to zero. 
\item Continuing in this manner, in each step, the preimage of $0$
  under the obstruction map defines homogeneous relations in terms of
  the elements 
  $t_1,\dots,t_\tau$ of a basis of $(T^1_{(X,x)})^\ast$, of a given
  order, which in the limit yield formal power series in
  \mbox{$\C[[\bt]]=\C[[t_1,\dots,t_\tau]]$}. If $J$ denotes the ideal in
  $\C[[\bt]]$ defined by these power series, the quotient
$\C[[\bt]]/J$ is the local ring of the base space of the (formal)
versal deformation. Then \mbox{$T^1_{(X,x)}=(\langle
  \bt\rangle/\langle \bt\rangle^2 )^\ast$} is the Zariski tangent
space to this base space. 
\end{itemize}
This method works for very general deformation functors having an obstruction
theory. A collection of methods and results from general obstruction theory 
can be found in \cite[Appendix C.2]{GLS07}.

\medskip
We give now a concrete descripti ofon the module $T^2_{(X,x)}$, containing the obstructions to lift a deformation  
from a fat point $(T,\bo)$ to an infinitesimally bigger one
\mbox{$(T'\!,\bo)$}.

Let \mbox{$\ko_{X,x}=\ko_{\C^n\!,\bo}/I$}, with \mbox{$I=\langle f_1,\dots,
    f_k\rangle$}. Consider a presentation of $I$,
$$ 0 \longleftarrow I \stackrel{\alpha}{\longleftarrow} \ko_{\C^n\!,\bo}^k
\stackrel{\beta}{\longleftarrow}  \ko_{\C^n\!,\bo}^\ell\,,\qquad \alpha(e_i)=f_i\,.$$
\mbox{$\Ker(\alpha)=\im(\beta)$} is the module of relations for
\mbox{$f_1,\dots,f_k$}, which contains the
\mbox{$\ko_{\C^n\!,\bo}$}-module of {\em Koszul 
  relations}\index{Koszul!relations}\index{Koszul}
$$ \Kos := \langle f_i e_j - f_je_i \mid 1\leq i < j \leq k\rangle\,, $$
$e_1,\dots,e_k$ denoting the standard unit vectors in $\ko_{\C^n\!,\bo}^k$.
We set $\Rel:=\Ker(\alpha)$ and note that $\Rel/\Kos$ is an
$\ko_{X,x}$-module: let \mbox{$\sum_{i} r_ie_i\in \Rel$}, then 
$$ f_j \cdot  \sum_{i=1}^k r_ie_i = f_j  \cdot  \sum_{i=1}^k r_ie_i -
\sum_{i=1}^k r_if_i e_j =  \sum_{i=1}^k r_i \cdot (f_je_i-f_ie_j) \in
\Kos\,. $$
Since \mbox{$\Kos\subset I\ko_{\C^n\!,\bo}^k$}, the inclusion \mbox{$\Rel\subset
  \ko_{\C^n\!,\bo}^k$} induces an $\ko_{X,x}$-linear map
$$ \Rel/\Kos \lra \ko_{\C^n\!,\bo}^k/I\ko_{\C^n\!,\bo}^k = \ko_{X,x}^k\,.$$

\begin{definition}
We define $T^2_{(X,x)}$\index{$t2xx$@$T^2_{(X,x)}$} to be the cokernel 
of $\Phi$, the $\ko_{X,x}$-dual of the latter map, that is, we have a defining
exact sequence for $T^2_{(X,x)}$:
$$
\Hom_{\ko_{X,x}}(\ko_{X,x}^k,\ko_{X,x}) \xrightarrow{\Phi}
\Hom_{\ko_{X,x}}(\Rel/\Kos,\ko_{X,x}) \to T^2_{(X,x)} \to 0\,.
$$
\end{definition}

The following proposition is proved in \cite[Proposition II.1.29]{GLS07} and \cite[Chapter 3]{St03}.
\begin{proposition}\label{prop:t2xx} Let $(X,x)$ be a complex space germ.
  \begin{enumerate}[leftmargin=*]
\itemsep2pt
  \item[(1)] Let \mbox{$j:(T,\bo)\hookrightarrow (T'\!,\bo)$}
    be an inclusion of fat points, and let $J$ be the kernel of the
    corresponding map of local rings
    \mbox{$\ko_{T'\!,\bo}\twoheadrightarrow \ko_{T,\bo}$}. 
Then there is a map, called the {\em
  obstruction map},\index{$ob$@$\ob$}\index{obstruction!map}
$$ \ob: \uDef_{(X,x)}(T,\bo)\lra  T^2_{(X,x)} \otimes_\C
J\,, $$ 
satisfying: a deformation
\mbox{$(i,\phi):(X,x)\hookrightarrow(\sX,x) 
  \to(T,\bo)$} admits a lifting
\mbox{$(i',\phi'):(X,x)\hookrightarrow(\sX',x) 
  \to (T'\!,\bo)$} (i.e.,
\mbox{$j^\ast(i',\phi')=(i,\phi)$}) iff
\mbox{$\ob\bigl([(i,\phi)]\bigr)=0$}. 

\item[(2)] If  $T^1_{(X,x)}$ is a finite dimensional $\C$-vector space 
    and if \mbox{$T^2_{(X,x)}=0$}, then the semiuniversal deformation
      of $(X,x)$ exists and has a smooth base
      space (of dimension \mbox{$ = \dim_\C
        T^1_{(X,x)}$}).\index{semiuniversal!deformation!base space} 
  \end{enumerate}
\end{proposition}
\noindent
Note that the obstruction map $\ob$\index{$ob$@$\ob$}\index{obstruction!map} is
a map between sets (without further 
structure) as $\uDef_{(X,x)}(T,\bo)$ is just a set.

\begin{definition}
We call $(X,x)$ {\em unobstructed} \index{unobstructed} if it has a semiuniversal  deformation with smooth base space.
\end{definition}
\noindent
Hence an isolated singularity is unobstructed if $T^2_{(X,x)}=0$, but $(X,x)$ may be unobstructed even if $T^2_{(X,x)} \neq 0$.
\medskip

If $(X,x)$ is a hypersurface or, more general, a complete intersection, then the Koszul relations are the only existing relations. Hence, in this case \mbox{$\Rel=\Kos$} and
\mbox{$T^2_{(X,x)}=0$}. In particular, isolated complete intersection singularities are unobstructed. 

\smallskip
Statement (2) of Proposition \ref{prop:t2xx} can be generalized by applying 
 Laudal's theorem (\cite[Theorem 4.2]{La79}), which relates the
 base of a 
formal semiuniversal deformation of $(X,x)$ with the fibre of a formal power
series map:

\begin{theorem}[Laudal]\index{Laudal}
 Let $(X,x)$ be a complex space germ such that \mbox{$T^1_{(X,x)}$},
 \mbox{$T^2_{(X,x)}$} are finite dimensional complex vector spaces. Then there
 exists a formal power series map  
$$ \Psi: T^1_{(X,x)} \lra  T^2_{(X,x)} $$
such that the fibre \mbox{$\Psi^{-1}(0)$} is the base of a formal semiuniversal
deformation of $(X,x)$.
\end{theorem}

\begin{corollary}\label{cor:T2xx}
 Let $(X,x)$ be a complex space germ such that \mbox{$T^1_{(X,x)}$},
 \mbox{$T^2_{(X,x)}$} are finite dimensional complex vector spaces, and let $(S,s)$
 be the base space of the semiuniversal deformation. Then
$$ \dim_\C T^1_{(X,x)}\geq \dim (S,s) \geq  \dim_\C T^1_{(X,x)} - \dim_\C
T^2_{(X,x)}\,,$$
and \mbox{$\dim (S,s)=\dim_\C T^1_{(X,x)}$} iff $(S,s)$ is
smooth. 
\end{corollary}

\noindent 
This corollary holds in a general deformation theoretic context 
(see \cite[Proposition C.2.3]{GLS07}).

The $\ko_{X,x}$-module $T^2_{(X,x)}$ contains the obstructions against 
smoothness of the base space of the semiuniversal deformation (if it
exists), but it may be strictly bigger. That is, in Corollary
\ref{cor:T2xx}, the dimension of $(S,s)$ may be strictly larger than
the difference \mbox{$\dim_\C T^1_{(X,x)} -\dim_\C T^2_{(X,x)}$},
as in the following example.

\begin{example}\label{ex:ratnorm}
{\em  Let us compute the full semiuniversal deformation of the cone
  \mbox{$(X,\bo)\subset (\C^5\!,\bo)$} over the rational normal curve of
  degree $4$, using {\sc Singular}.  We get \mbox{$\dim_\C T^1_{(X,\bo)}=4$} and
  \mbox{$\dim_\C T^2_{(X,\bo)}=3$} and that the semiuniversal base space has dimension 3.
  The total space of the semiuniversal
  deformation has $4$ additional
  variables $A,B,C,D$ (in the ring \texttt{Px}), the unfolding of the
  $6$ defining equations of $(X,\bo)$ is given by the ideal \texttt{Fs}
  and the base space, which is given by the ideal \texttt{Js} in
  $\C\{A,B,C,D\}$, is the union of the $3$-plane \mbox{$\{D=0\}$} and
  the line \mbox{$\{B=C=D-A=0\}$} in \mbox{$(\C^4\!,\bo)$}:

\begin{small}
\begin{verbatim}
LIB "deform.lib";
ring R = 0,(x,y,z,u,v),ds;
matrix M[2][4] = x,y,z,u,y,z,u,v;
ideal I = minor(M,2);   // rational normal curve in P^4
vdim(T_1(I));
//-> 4
vdim(T_2(I));
//-> 3

list L = versal(I);     // compute semiuniversal deformation
//-> // ready: T_1 and T_2
//-> // start computation in degree 2.
//-> .... (further output skipped) .....

def Px=L[1];
show(Px);
//-> // ring: (0),(A,B,C,D,x,y,z,u,v),(ds(4),ds(5),C);
//-> // minpoly = 0
//-> // objects belonging to this ring:
//-> // Rs                   [0]  matrix 6 x 8
//-> // Fs                   [0]  matrix 1 x 6
//-> // Js                   [0]  matrix 1 x 3

setring Px;
Fs;                     // equations of total space
//-> Fs[1,1]=-u2+zv+Bu+Dv
//-> Fs[1,2]=-zu+yv-Au+Du
//-> Fs[1,3]=-yu+xv+Cu+Dz
//-> Fs[1,4]=z2-yu+Az+By
//-> Fs[1,5]=yz-xu+Bx-Cz
//-> Fs[1,6]=-y2+xz+Ax+Cy
Js;                     // equations of base space
//-> Js[1,1]=BD
//-> Js[1,2]=-AD+D2
//-> Js[1,3]=-CD
\end{verbatim}
\end{small}

\noindent
Hence, the semiuniversal deformation of $(X,\bo)$ is given by \mbox{$(\sX,\bo)
  \to (S,\bo)$}, induced by the projection
onto the first factor of \mbox{$(\C^4\!,\bo)\times (\C^5\!,\bo)$},
$$ (\C^4\!,\bo)\times (\C^5\!,\bo)\supset V(\texttt{Fs})=(\sX,\bo) \to
(S,\bo)=V(\texttt{Js})\subset (\C^4\!,\bo)\,.$$
Note that the procedure \texttt{versal} proceeds
by lifting infinitesimal deformations to higher and higher order (as described
in the proof of Proposition \ref{prop:t2xx}). In general, this process may be infinite
(but \texttt{versal} stops at a predefined order). However, in many examples,
it is finite (as in the example above).

We can further analyse the base space of the semiuniversal deformation by
decomposing it into its irreducible components.
\begin{small}
\begin{verbatim}
ring P = 0,(A,B,C,D),dp;
ideal Js = imap(Px,Js);
minAssGTZ(Js);
//-> [1]:
//->    _[1]=D
//-> [2]:
//->    _[1]=C
//->    _[2]=B
//->    _[3]=A-D
\end{verbatim}
\end{small}

\noindent
The output shows that the base space is reduced (the primary and prime
components coincide) and that it has two components: a hyperplane and a
transversal line.}
\end{example}

Furter developments:
Abstract deformation theory, basically governed by the Schlessinger's conditions, has been further developed towards ``Derived Deformation Theory'' (cf.\cite{Lu09})
following the general trend in agebra and algebraic geometry to make everything ``derived''. While derived algebraic geometry has already become part of the mainstream, this does not yet apply to ``Derived Singularity Theory''.
 \bigskip

\section{Smoothing of Singularities}
\setcounter{equation}{0}
\label{sec:smoo}

\subsection{Rigidity and Smoothability } \label{sec:rigsmo}
We give a brief review of some well-known results on the question of the smoothability and rigidity of singularities. Recall that a complex space germ 
$(X, x)$ with isolated singularity is not obstructed iff the semiuniversal deformation base space is smooth.

\begin{definition}
(1) A singularity $(X,x)$ is called {\em
  rigid\/}\index{rigid!singularity}\index{singularity!rigid}  if any
deformation of $(X,x)$ over some base space $(S,s)$ is 
{\em trivial}\index{deformation!trivial}\index{trivial!deformation},
that is, isomorphic to the {\em product deformation}
$$
(X,x)\stackrel{i}{\hookrightarrow}(X,x)\times(S,s)\stackrel{p}{\to}(S,s)
$$
with $i$ the canonical inclusion and $p$ the second projection.  

(2) $(X,x)$ is is called {\em smoothable \/}\index{smoothable!singularity}\index{singularity!smoothable}if 
 there exists a 1-parametric deformation $\phi:(\sX,x) \to (\C,0)$
of $(X,x)$ such that for $t\in \C \setminus \{0\}$ sufficiently close to $0$ the fibre $\sX_t = \phi^{-1}(t)$ is smooth.
\end{definition}

\noindent
Rigid singularities are unobstructed. Smooth germs are rigid (by the implicit function theorem) and smoothable. For non-smooth singularities the notions are opposite to each other. If $(X,x)$ has an isolated singularity then $(X,x)$ is rigid iff the semiuniversal base is a reduced point while $(X,x)$ is smoothable iff the semiuniversal base has a positive dimensional irreducible component over which the generic fibre is smooth.  Such a component is called a {\em smoothing component}.\index{smoothing!component}

\begin{proposition}
A complex space germ is rigid\index{rigid!singularity} iff
\mbox{$T^1_{(X,x)}=0$}. 
\end{proposition}

\begin{proof}
$(X,x)$ is rigid iff the semiuniversal deformation exists and
consists of a single, reduced point. By Lemma \ref{lem:Kodaira-Spencer
  map}, together with 
the existence of a semiuniversal deformation for germs with
\mbox{$\dim_\C T^1_{(X,x)}<\infty$}, this is equivalent to
\mbox{$T^1_{(X,x)}=0$}.  
\end{proof}

   The existence of rigid singularities in small dimension 
  is still an open problem.  Assuming that the singularities are not smooth one may conjecture:

\begin{conjecture}
There exist no rigid rigid fat points, no rigid reduced curve singularities and no rigid
normal surface singularities.   
\end{conjecture}

\begin{example}\label{exa:semiuni def}
{\em  (1) The simplest known example of an equidimensional
    (non-smooth) rigid singularity $(X,\bo)$ is the union of two planes
    in \mbox{$(\C^4\!,\bo)$}, meeting in one point  and defined by \mbox{$\langle
  x,y\rangle\cap\langle z,w\rangle$} 
    (given by the ideal
    \texttt{I} in the ring \texttt{R} below). 
    
\noindent    
(2)  The product
    \mbox{$(X,\bo)\times (\C,0)\subset (\C^5\!,\bo)$} (given by the ideal
    \texttt{I} in the ring \texttt{R1 below}) has a non-isolated
    singularity but is also rigid (hence, has a semiuniversal
    deformation). We prove these statements using  {\sc Singular}
    (\cite{DGPS18}):
\begin{small}
\begin{verbatim}
LIB "deform.lib";
ring R = 0,(x,y,u,v),ds;
ideal I = intersect(ideal(x,y),ideal(u,v));
vdim(T_1(I));         // result is 0 iff V(I) is rigid
//-> 0
ring R1 = 0,(x,y,u,v,w),ds;
ideal I = imap(R,I); 
dim_slocus(I);        // dimension of singular locus of V(I)
//-> 1
vdim(T_1(I));
//-> 0
\end{verbatim}
\end{small}
(3) An even simpler (but not equidimensional) rigid singularity 
  is the union of the plane \mbox{$\{x=0\}$} and the line
  \mbox{$\{y=z=0\}$} in \mbox{$(\C^3\!,\bo)$}. This can be checked either
  by using {\sc Singular} as above, or, without computer, by
  showing that the map $\beta$ in Proposition \ref{prop:t1xx} is
  surjective.   }
\end{example}
 
Let us first recall some known results on rigidity. In 1964 Grauert and Kerner \cite{GK64} generalized Thom's example (see below) 
 and showed that the Segre cone over 
$\P^r \times \P^1$ in $\P^{2r + 1} (r \geq 2)$ is rigid and gave
thus the first example of a (non-smooth) rigid singularity. Further examples of rigid singularities are due to Schlessinger (isolated quotient singularities of dimension 
$\geq 3$ \cite{Sc71}, and to Rim (e.g. the one-point union of two copies of $(\C^n, 0)$ in $(\C^{2n},0$) for $n \geq 2$, \cite{Ri72}). Examples
of singularities that are not deformable into rigid singularities (so-called ``generic singularities'') are due to Schlessinger \cite{Sc73} ($\dim \geq 3$) and Mumford 
\cite{Mu73}  ($\dim \geq 2$) (c.f. also \cite{Pi74}).  
By Herzog \cite{He79} one-dimensional, almost-complete intersections are not rigid.  It is also known that monomial (i.e., irreducible, quasihomogenous)
curves are not rigid (\cite [5.12]{RV77}, \cite[3.1.2]{Bu80}). To date, no examples of rigid curve singularities are known; it is conjectured that they do not exist.
A detailed discussion of rigid and smoothable singularities together with references up to 1973 can be found in Hartshorne \cite{Ha74}, where also topological conditions for smoothability are derived.
\medskip

 The  question  whether a singularity $(X,x)$ is smoothable is among others interesting because the smooth nearby fibre is an important topological object associated to the singularity that has been (and still is) a continuous subject of research (see section \ref{sec:top}).
Classically, it was even suspected that all singularities are smoothable. 
In 1909 Severi postulated that each algebraic variety with arbitrary singularities
should be the limit of a family of nonsingular algebraic manifolds (\cite[p.45]{Sev09} and in  \cite[p.355]{Sev21} for curves). 
In fact he conjectured that an irreducible curve can be smoothed in a family of curves with constant degree and arithmetic genus, i.e., in a flat family. 
It was a guiding principle of Severi in \cite{Sev21} to obtain statements about singular curves from their smoothing. 

Of course, hypersurfaces (e.g. plane curves) are smoothable but Severi's general postulate turned out to be wrong.  The first example of a non-smoothable singularity, the cone apex of the Segre embedding of 
$\P^2 \times \P^1$ in $\P^5$, was found by R. Thom in 1957. Thom gave topological reasons for non-smoothability; his argument was reproduced and worked out in 1974 in \cite{Ha74} (for a strengthening see Theorem \ref{thm:topsmcond}). 
The first rigorous and pure
algebraic proof was published anonymously in \cite{XXX57} in 1957 (according to Thom, the author is A. Weil).
The author shows that the projective closure of Thom's example in $\P^6$ can not be smoothed in $\P^6$ (which is however weaker than abstract non-smoothability, cf. \cite[2.12]{Pi78}). Rees and Thomas \cite{RT78a, RT78b} developed Thom's idea further and found refined cobordism invariants of the neighborhood boundary of an isolated singularity $(X,x)$ as a necessary condition for smoothability. They gave also
further examples of non-smoothable singularities. 
Other conditions have been found by Sommese \cite{So79}.

The conjecture of Severi that every irreducible projective curve is smoothable  has not been doubted for a long time (cf. \cite{Ri72} and \cite{De73}) until Mumford's, and later Pinkham's examples appeared in 1973.
Mumford showed in \cite{Mu75} that non-smoothable irreducible curve singularities exist. Reducible examples, related to $r$ straight lines in $\C^n$ through $0$ in general position, were first found by Pinkham \cite{Pi74}. We consider these examples and generalizations in  section \ref{sec:curves}. 
\medskip

Further results on smoothability:
\begin{enumerate} 
\item Complete intersections are smoothable (by Sard's theorem) and 
non-obstructed (\cite{Tj69}).
\item A {\em determinantal singularity}\index{determinantal singularity} $(X,x)$ is given by the $t \times t$ minors of an $r \times s$ matrix with entries holomorphic functions in an open subset $U \subset \C^N$, such that $(X,x)$ has codimension $(r-t+1)(s-t+1)$ in $\C^N$. If $(X,x)$ is an isolated determinantal singularity and
$2 \leq t \leq r \leq s$, then $(X,x)$ is smoothable if $\dim(X,x)<s+r- 2t+3$ (\cite[6.2]{Wa81}). 
\item  In particular, if $(X,x)$ is Cohen-Macaulay of codimension{\footnote { the codimension of $(X,x)$  is $\codim (X, x) = \edim(X,x) - \dim(X,x)$ with $\edim(X,x)  = \dim_\C \fm/\fm^2$ the embedding dimension.}} 2
(and hence determinantal) then $(X,x)$ is smoothable provided $\dim(X,x) \leq 3$ (Schaps, \cite{Sch77}). 
Note that the semiuniversal base is smooth for Cohen-Macaulay singularities in codimension  2 without any restriction on the dimension (Schlessinger, thesis, and \cite{Sch77}).
\item Moreover, an isolated {\em Pfaffian singularity}\index{Pfaffian singularity}
 $(X,x)$, defined by the $2m \times 2m$ Pfaffians of a skew-symmetric 
 $(2n + 1) \times (2n + 1)$ matrix of holomorphic functions is smoothable if $\dim(X,x) < 4(n - m) +7$ (\cite[6.3]{Wa81}). 
\item  An irreducible Gorenstein singularity of $\codim (X, x) \leq 3$ 
(which is Pfaffian) is 
smoothable if $\dim (X, x) \leq 6$ and has a smooth semiuniversal base (Waldi, \cite{Wal79}). 
\end{enumerate}

Let us look at dimensions $\leq 2$: 
\begin{enumerate}
\item  The first examples of nonsmoothable normal surface singularities were found by Mumford and later by Pinkham, see section \ref{sec:surfaces} where the case of surface singularities is treated in more detail. 
%
%
\item For  curves the following is known. In 1975 Mumford gave the first example of non-smoothable curve singularities, using similar ideas as Iarrobino for his examples of non-smoothable fat points.
\item Since reduced curve singularities are Cohen-Macaulay, we get that reduced curves
in $(\C^n, 0)$, $n \leq 3$, and reduced, irreducible Gorenstein curves
in $(\C^n, 0)$, $n \leq 4$, are smoothable and not obstructed.  In \cite{RV77}
it is shown that negatively graded monomial curves are smoothable. For reduced quasihomogeneous curves this is not true by Pinkham’s examples.
Explicit examples of non-smoothable monomial curves were first found by Buchweitz in \cite{Bu80}.

 It is interesting to note that the curves of Mumford and Pinkham are not smoothable, since the dimension of the base space of the semi-universal deformation is ``too large''. There is no curve singularity known whose semi-universal base has a smaller dimension than it would have by Deligne's formula 
(see section \ref{sec:curves}) if it were smoothable.

\item Fat points in $\C^2$ are smoothable (c.f. \cite{Bri77}). 
\item First examples of non-smoothable points in  $\C^n, n\geq 3,$ were found by Iarrobino \cite{Ia72}. For an overview on the Hilbert scheme of points, i.e. the deformation theory of a collection of fat points in some projective space (until 1987), see  \cite{Ia87}. 
\item Since then quite some work concerning smoothability of Artin algebras was done, in particular more examples and methods that show non-smoothability have been found.  E.g. Shafarevich proves in \cite{Sh90} that a large number of fat points are in fact non-smoothable (as it had been expected).
\end{enumerate}
\bigskip

\subsection{Smoothing of Surface Singularities}
\setcounter{equation}{0}
\label{sec:surfaces}
Smoothability has been an important part of the deformation theory of normal surface singularities. For any smoothing one has a smooth Milnor fibre, a key topological object that has been intensively studied. 
It started with Pinkham's examples in \cite{Pi74, Pi78} and was continued by Wahl \cite{Wa81, Wa82}, both studied deformations and possible smoothings of weighted homogeneous normal surface singularities. 
\medskip

An arbitrary singularity $(X,x)$ is  {\em weighted homogeneous}\index{weighted homogeneous}  if $\ko_{X,x}$ is a graded algebra $\ko_{X,x} = \C\{x_1,...,x_n\}/I$, where the $x_i$ have   positive weights, $wt(x_i)= a_i >0$, and
$I$ is a graded ideal, which is generated by (weighted) homogeneous polynomials $f_i$.   Then $(X,x)$ admits a good $\C^*$-action 
$\lambda \cdot (x_1,...,x_n) = (\lambda^{a_1}x_1,...,\lambda^{a_n}x_n)$. 
We call $(X,x)$ {\em quasihomogeneous}\index{quasihomogeneous} if it is analytically isomorphic to a weighted homogeneous singularity.

The following result, a complement to Grauert's Theorem \ref{thm:grauert}, was proved by Pinkham in \cite{Pi74, Pi78}.

\begin{theorem}[Pinkham]\index{Pinkham} \label{thm:pinkham}
\begin {enumerate}
\item A weighted homogeneous isolated singularity $(X,x)$ 
admits a semiuniversal deformation 
$\phi: (\sX,x) \to (S,s)$ such that the $\C^*$-action extends to $(\sX,x)$ and 
$(S,s)$ with $\phi$ equivariant. 
\item Any equivariant deformation 
$(\sY,y) \to (T,t)$ of $(X,x)$ can be induced from $(\sX,x) \to (S,s)$ via an equivariant base change morphism $\varphi : (T,t) \to (S,s)$. 
\item For any equivariant deformation $(\sY,y) \to (T,t)$ choose homogeneous generators $t_j$ of the maximal ideal of $\ko_{T,t}$ and set 
$(T^-,t) = V\{t_j \,| \,wt(t_j) < 0 \}$ (resp. $(T^0,t) = V\{t_j \,| \,wt(t_j) \neq 0 \}$,
resp.  $(T^+,t) = V\{t_j \,| \,wt(t_j) > 0 \}$).

Then the equivariant morphism $\varphi : (T,t) \to (S,s)$ of (2) restricts to $\varphi^-: (T^-,t) \to (S^-,s)$ (resp. $\varphi^0: (T^0,t) \to (S^0,s)$, resp. $\varphi^+: (T^+,t) \to (S^+,s)$). 

\end {enumerate}
\end{theorem} 

\begin{proof} We sketch only the proof of (1) following  \cite{Pi74, Pi78}, who proves the statement in the setting of formal deformation theory.   The analytic version for complex germs follows from an appropriate modification of  \cite[Proposition 1]{GH74} taking care of the $\C^*$-action.  

Choose homogeneous generators $f_1,...,f_k$ of $I$ of (weighted) degree $d_i$. We use the exact sequence
 $$\Theta_{\C^n,\bo}
  \otimes_{\ko_{\C^n,\bo}}\ko_{X,\bo}\stackrel{\bet}{\longrightarrow}
    \Hom(I/I^2,\ko_{X,\bo})=T^1_{X,\bo/\C^n\!,\bo}\longrightarrow T^1_{X,\bo}\longrightarrow 0\,, $$
from   Proposition \ref{prop:t1xx}, where the first module is graded by setting
$wt(\frac{\partial}{\partial x_j}) = -a_i$.
By Remark \ref{rmk:t1xx}
very element in $T^1_{X,\bo/\C^n\!,\bo}$
 is  given by \mbox{$f_i+\eps g_i$}, $i=1,\dots,k$, i.e., given by a tupel
  $G= (g_1,\dots,g_k)$ with $g_i \in \ko_{\C^n,0}$. 
  We define $G$ to be homogeneous of  degree $\nu$ if 
  $g_i$ is homogeneous of degree $\nu+d_i$, thus imposing a grading on   
  $T^1_{X,\bo/\C^n\!,\bo}$.  It follows that $\beta$ is homogeneous since 
  $\beta(\frac{\partial}{\partial x_j})=
  (\frac{\partial f_1}{\partial x_j},...,\frac{\partial f_k}{\partial x_j})$. Therefore 
  $\coker (\beta) = T^1_{X,\bo}$ is graded and $T^1_{X,\bo}$ 
  decomposes into graded pieces
$$T^1_{X,\bo} = \sum_{\nu \in \Z} T^1_{X,\bo}(\nu).$$ 
We chose  homogeneous  elements 
$G_j = (g_1^j,...,g_k^j) \in T^1_{X,\bo/\C^n\!,\bo}$ with
$\deg(G_j) = \nu_j$, mapping to a homogeneous basis of $T^1_{X,\bo}$, $j=1,...,\tau$.
We choose new variables $t =(t_i, ..., t_\tau)$ and set
$$(f'_1,...,f'_k) = (f_1,...,f_k) + \sum_{j=1,...,\tau} t_j(g_1^j,...,g_k^j) \ \ (\text{mod } \fm^2),$$
$\fm$ the maximal ideal of $\C\{t_1,...,t_\tau\}$. Then $(f'_1,...,f'_k)$  defines a first order deformation of $(X,\bo)$ with total space 
$(\sX',\bo) \subset(\C^n \times \C^\tau,\bo)$ defined by $\langle f'_1,...,f'_k\rangle 
 \subset \C\{x,t\}$ 
over $(S',0) =V(\fm^2) \subset (\C^\tau,0)$ as base space. Giving  $t_j$ the weight $- \nu_j$ 
then $f'_j$ is homogeneous of degree $d_j$ and $(\sX',\bo) \to (S',0)$ is an equivariant deformation of first order. 

Now we continue as in section \ref{sec:obs} and lift the first order deformation to second order but in an equivariant way. Continuing by induction, we get finally
an equivariant semiuniversal deformation of $(X,\bo)$. 

If the $G_j$ are a system of generators of  $T^1_{X,\bo}$, we get an equivariant versal deformation.
The equivariant base change property is proved in a similar way by induction.
\end{proof} 

\begin{remark}
(1) It follows from the proof, that the base space of the semuniversal deformation is given by a subgerm $(S,0) \subset (\C^\tau,0)$, $\tau =  \dim_\C T^1_{X,\bo}$,
defined by some homogeneous ideal in
$\C\{t_1,...,t_\tau\}$, $wt(t_j) \in \Z$
(note that the signs of the weights of the variables $t_j$ are opposite to signs of the weights of the tangent vectors). The total space of the semiuniversal defromation is then a subgerm $(\sX,\bo) \subset (\C^n,\bo)\times (S,0)$ and $\phi: (\sX,\bo) \to (S,0)$ is the projection.
$(\sX,\bo)$ is defined by a homogeneous ideal 
$J \subset \ko_{(\C^n,\bo)\times (S,0)}$ generated by homogeneous power series
$$F_j(x,t) = f_j(x) + g_j(x,t) \in  \C\{x_1,...,x_n,t_1,...,t_\tau\}, \    g_j(x,0)=0.$$

(2) The restriction  
$\phi^-: (\sX^-,\bo) \to (S^-,0)$  of $\phi$ is defined by $f_j(x) + g_j(x,t)$ 
with $\deg(g_j)\geq \deg(fj)$ 
and any deformation which is induced form a map to $(S^-,0)$
is called a {\em deformation of non-positive weight}\index{deformation!of non-positive weight}; if it is induced form a map to  $(S^-,0)\cap(S^0,0)$  (i.e. $\deg{g_j} > \deg(fj)$) we call it a {\em deformation of negative weight}\index{deformation!of negative weight}.
Similarly we consider the restriction  $\phi^+: (\sX^+,\bo) \to (S^+,0)$ and speak of
 {\em deformations of non-negative } resp. of {\em positive weight}.\index{deformation!of positive weight}

(3) It is easy to see that $(X,x)$ cannot have smoothings of non-positive weight (consider the Jacobian of $(F_1,...,F_k$ and use that the total space of a 1-parametric smoothing has an isolated singularity). $(X,x)$ may have smoothings of positive weight, but these are rare as we shall see.
\end{remark}

Let us now recall the main results about smoothability for a normal surface singularity $(X,x)$. Besides complete intersections the following is known:

\begin {enumerate}
\item {\em Rational singularities}\index{singularity!rational} and especially {\em quotient singularities} are always smoothable over the  {\em Artin component},
\index{Artin component} i.e. the component of of the semiuniversal base corresponding to deformations of $(X,x)$, induced by blow down, from deformations of the resolution of $(X,x)$ (cf. Artin \cite{Ar74}; see also \cite[Proposition 6.10]{Pi78}). 

\item A {\em normal surface singularity in $(\C^4,0)$} is smoothable with a smooth semi-universal base  space since it is Cohen-Macaulay in codimension 2.

\item Let $(X,x)$ be a {\em simple elliptic singularity},\index{singularity!simple elliptic} i.e. the exceptional divisor of the minimal resolution consists of one elliptic curve with selfintersection number  $-d$. Note that $d$ is the multiplicity $m$ of 
$(X,x)$, except for $d= 1$ where $m= 2$.
Then  $(X,x)$ is smoothable if and only if $m \leq 9$ (Pinkham \cite{Pi74}).

\item Let $(X,x)$ be a {\em cusp singularity}\index{singularity!cusp}  where the exceptional curve of the minimal resolution consists of a cycle of $r$ rational curves
meeting transversally.  Let $m$ denotes again the multiplicity. 
Then $(X,x)$ is smoothable if $m^2 - m < r$ and is not smoothable if $m > r + 9$ (Wahl  \cite[5.6]{Wa81}, \cite[5.12]{Wa80}).

\item Looijenga proved in \cite{Lo81} that whenever a cusp singularity is smoothable, the minimal resolution of the dual cusp is an anticanonical divisor of some smooth rational surface. He conjectured the converse. The conjecture was proved by Gross, Hacking, and Keel \cite{GHK15} using methods from mirror symmetry. For an  alternative proof see \cite{En18}. 

\item If $(X,x)$ is a {\em Dolgachev (or triangular) singularity}, then $(X,x)$ is not smoothable if the multiplicity is $\geq 14$ (\cite{Wa82}; see also \cite{Lo83}).

\item Note that the last three classes are {\em minimally elliptic singularities}\index{singularity!minimally elliptic}
 in the sense of Laufer (i.e. Gorenstein and $h^1(\tilde X, \ko_{\tilde X}) = 1$ for any resolution $\tilde X$ of $X$). 
Karras proved in \cite{Ka83} that each minimally elliptic singularity $(X,x)$ can be deformed into a simple elliptic singularity with the same multiplicity $m$. 
Hence a minimally elliptic singularity can be smoothed if $m \leq 9$.
\end {enumerate}

%
\medskip

Important obstructions against smoothability of an isolated singularity come from globalizing the smoothing. A smoothing of the globalized singularity (a projective variety)  provides a smooth projective variety in some projective space with properties (coming from the singularity) that cannot exist. The following theorem uses this method and is due to Pinkham \cite[Theorem 7.5]{Pi74}.

\begin{theorem}[Pinkham] 
Let $C \subset \P^n$ be a smooth projectively normal curve of genus $g \geq 1$ and degree $d \geq 10$  if $g = 1$ or $d \geq 4g + 5$ if $g \geq 2$. 
Let $X \subset \C^{n+1}$ denote the affine cone over $C$. Then the singularity 
$(X,0)$ is not smoothable. 
\end{theorem}

The proof makes use of the following Theorem of \cite[Theorem 4.2]{Pi74} that proves globalization for cones.

\begin{theorem}[Pinkham] 
Let $Y \subset \P^n$ be a nonsingular, projectively normal subvariety of dimension 
$\geq 1$. Let $X$ be the affine cone over $Y$ in $\C^{n+1}$  and let $\overline X$ be its projective closure in $\P^{n+1}$. Assume that the homogeneous singularity 
$(X,0)$ has negative grading (i.e., $T^1_{X,0}(\nu)=0$ for all $\nu >0$). 
Then any deformation of $(X,0)$ lifts to an
embedded deformation of $\overline X \subset \P^{n+1}$.  
More precisely, the morphism of deformation functors 
$\uDef_{\overline X / \P^{n+1}} \to \uDef_{(X,0)}$ is smooth.
\end{theorem}

\begin{remark} 
Pinkham proves the theorem only for infinitesimal deformations, i.e., for deformations  over fat points. Let us see how this implies theorem for deformations over arbitrary complex space germs: $\overline  X \subset \P^{n+1}$ and  $(X,0)$ have both a convergent semiuniversal deformation. Pinkham's result implies that the induced morphism of the completion of the local rings of their base spaces is smooth, i.e. flat with smooth fibre. This implies that the morphism of their analytic local rings is smooth since completion is faithfully flat.
\end{remark}

The following theorem is due to Pinkham \cite[6.14]{Pi78} and Wahl  \cite[3.9]{Wa82}.

\begin{theorem}[Pinkham, Wahl]\label{thm:piwa}
Let $(X,x)$ be a normal Gorenstein surface singularity with weighted dual graph 
of the minimal resolution being star-shaped with $n$ arms and a central curve, 
where the end-vertex the $i$-th arm corresponds to a smooth rational curve of 
self-intersection $-b_i$ ($n\geq 3$, $b_i \geq 2$). 
If  $(X,x)$ is smoothable, then 
$$\sum_{1\leq i \leq n}(b_i -1) \leq 19.$$
\end{theorem}

The method of Pinkham is by globalizing and using hyperplane sections to find
obstructions for smoothings. He proves the above bound for smoothings of negative weight where the negativity assumption is used to globalize the smoothing. 
Wahl showed 
(\cite[3.8]{Wa82}), under the assumptions of the theorem, that any deformation of 
$(X,x)$, in particular any smoothing, can be globalized in the following sense:

Since $(X,x)$ is weighted homogeneous with isolated singularity, it has an affine 
representative $X\subset \C^n$ with $x$ as its only singularity. Let $\overline X$ be 
the projective closure in the corresponding weighted projective space. Then any deformation of the projective variety $\overline X$ induces a deformation of $X$ and Wahl shows that the induced functor of 
deformation classes $\uDef_{\overline X} \to \uDef_{X}$ is smooth.
\medskip

Later Looijenga proved (\cite[Appendix]{Lo86}) that any smoothing of an 
arbitrary isolated singularity $(X,x)$ can be globalized:

\begin{theorem}[Looijenga]\label{thm:loo}
Let $f: (\sX,x) \to (\C,0)$ be a smoothing over $(\C, 0)$ of an isolateded 
singularity. Then there is a flat projective morphism $F: \sY \to \C$, a point 
$y \in Y =F^{-1}(0)$ and an isomorphism $h:(\sX, x) \to (\sY, y)$ 
such that $F \circ h = f$ and $F$ is smooth along  $Y \setminus \{y\}$.
\end{theorem}
\medskip

Wahl’s paper \cite{Wa81} contains several conjectures which have all been proved shortly after. The above globalization property implies that Wahl's Theorem 3.13 holds for any smoothing of a normal surface singularity. The same is true for his Corollary 4.6 due to the results of Looijenga and the author in \cite{GL85}, while Theorem 4.10 is valid for any smoothing of a Gorenstein surface singularity. The other conjectures made in \cite{Wa81} follow from the results of Steenbrink in \cite{Ste83} and 
Steenbrink and the author in \cite{GS83}. See Sections \ref{sec:top} and \ref{sec:smoothcomp} for a treatment of these conjectures.
\medskip  

Since the 1990'th many further examples of smoothable and non-smoothable singularities were found (a search for ``smoothable'' in zbMATH (Zentralblatt) or Mathematical Reviews lists about 300 articles), often in the global setting for projective varieties and as a result of research on other questions. Moreover, the smoothability assumption is often used in proofs. For a treatment of (formal) smoothing of singularities in the deformation theoretic setting of schemes see \cite[Section 29]{Ha10}.
\bigskip

\subsection {Topology of the Milnor Fibre}\label{sec:top}
The main object of research for smoothable surface singularities $(X,x)$ is the topology of the Minor fiber. For the classical theory of the Milnor fibration and related topics we refer to the textbooks by Milnor 
\cite{Mi68} (hypersurfaces), Looijenga \cite{Lo13} (complete intersections), and Seade \cite{Sea06} (real singularities and index theorems).
For a computational approach to topological invariants of hypersurfaces we refer to \cite{Di92}.

In general one calls the generic fibre of any 1-parametric deformation the {\em Milnor fibre of the deformation}. To speak about the topology we need to choose special neighbourhoods.
\medskip

Let $(X,x) \subset (\C^N,x)$ be an arbitrary singularity. We
consider a morphism $\phi: (\sX,x) \to (S,s)$  with $(S,s) \subset (\C^k,s)$ and $(X,x) =(\phi^{-1}(s),x)$. We may assume that $\phi$ is embedded, 
i.e., $(\sX,x)$ is a closed subgerm of $(\C^N,x)\times(S,s)$ and $\phi$ the projection to the second factor. Let $U \subset \C^N\times \C^k$ be an open neighbouhood of $(x,s)$, 
$\widetilde\sX \subset U$ closed and $\phi:\widetilde \sX \to S$ a representative of the germ $\phi$. We choose now a special representative. 

\begin{definition}
In this situation let $B_\eps$ be an open ball of radius $\eps$ around $x$ in $\C^N$ 
and $\overline B_\eps$ the closed ball.
Let $S_\del$ be the intersection of $S$ with an open ball $D_\delta$  of radius $\del$ around $s$ in $\C^k$ with  $\overline{B}_\eps \times D_\delta \subset U$ and $0 < \del \ll \eps$ sufficiently small. Then 
$$\phi: \sX := \widetilde\sX \cap B_\eps \times S_\delta \to S_\delta=: S $$
is called a {\em good representative}\index{good representative} of $\phi$. The  fibres $\sX_t =  \phi^{-1}(t)$, $t\in S$, are contained in a fixed small ball
$B=B_\eps$ also called a {\em Milnor ball}.
Moreover,
$$\partial \sX_t := \widetilde \sX_t \cap \partial\overline B, \ t \in S, $$
 is the boundary of the closed fibre 
 $\overline \sX_t =  \widetilde \sX_t \cap \overline{B}$ and 
$\partial X = \partial \sX_0$
 is called the {\em neighbourhood boundary}\index{neighbourhood boundary} of  
 $(X,x)$. 
 
 If $(S,s) = (\C,0)$ we always write $\phi: \sX \to D $ for the good representative and call 
 $$F:= \sX_t , \ t \in D \setminus \{0\},$$
the {\em Milnor fibre}\index{Milnor!fibre}\index{fibre!Milnor} of the 1-parametric deformation $\phi$ and   $X = \sX_0$ the {\em special fibre}\index{fibre!special}.

\end{definition}

$\sX$ and all fibres $\sX_t$ are Stein complex spaces while $\partial \sX_t$ is a compact real algebraic subvariety of the sphere $\partial\overline{B}_\eps$.
The Milnor fibre depends of course in general on $\phi$ but for a given 1-parametric deformation its topological type is indpendent of $t \neq 0$ (cf. Theorem \ref{thm:le} below). 

The following lemma is certainly well known to specialists, but because of missing an explicit reference, I like to sketch a proof (thanks to H. Hamm).

\begin{lemma}\label{lem:strat}
If  $(X,x)$ is an isolated singularity then there are only  finitely many topologically different Milnor fibres for all deformations $\phi : \sX \to D$ of  $(X,x)$ . 
\end{lemma}

\begin{proof} Let $\phi: \sX \to S$ be a good representative 
 of the semiuniversal deformation  of $(X,x)$. Since $(X,x)$ has an isolated singularity,  the {\em singular} or {\em critical locus}\index{critical locus}\index{singular locus}  of $\phi$, 
 $$C(\phi) := \{y\in \sX | \sX_{\phi(y)} \text{ is singular at } y\} $$
 is finite over $S$ and the {\em discriminant}\index{discriminant} of $\phi$,
$$\Del(\phi) := \phi(C(\phi))$$ 
is an analytic subset of $S$. 
Consider now a {\em proper} representative 
$$\overline\phi: \overline \sX := \widetilde\sX \cap \overline B_\eps \times S \to S.$$
The fibres $\overline \sX_t$ of $\overline\phi$ meet $\partial\overline{B}_\eps$ transversally such that all boundaries $\partial \sX_t$ are differentiable manifolds.

Choose a Whitney stratification of $S$ such that $\Delta$ is a union of strata. The restriction $\phi: C \to \Delta$ has a Whitney stratification that refines this stratification of $\Delta$  (see \cite[I.1.7 Theorem, p. 43]{GM88}).
If one adds to the stratification of $C$ the
strata  $\overline{\phi}\,^ {-1}(T) \setminus C$, $T$ a stratum of $S$, 
one gets a stratification of $\sX$.
With this stratification and that of $S$ one obtains a Whitney stratification of the proper map $\overline \phi$ with finitely many strata. For every stratum $T$ of $S$ the restriction
$\overline{\phi}\,^ {-1}(T) \to T$ is a proper stratified submersion.
According to Thom's isotopy theorem (\cite[I.1.5 Theorem, p. 41]{GM88}) the latter 
defines a topological fiber bundle and the topological type of 
$\overline{\phi}\,^ {-1}(t),  t\in T$, is therefore independent of $t$.
\end{proof}

We are mainly interested in isolated singularities but let us first recall the following general result from \cite{Gre17}.

\begin{theorem}[Bobadilla, Greuel, Hamm] \label{thm:connected}
Let $\phi : (\sX, x) \to (\C, 0)$ be a morphism of complex germs and $\phi : \sX \to D$ a good representative with special fibre $X$ and Milnor fibre $F$.
\begin{enumerate}
\item If $(\sX,x)$ is irreducible and $X$ generically reduced then  $F$ is irreducible.
\item Let $(\sX,x)$ be reducible with irreducible components $(\sX_i,x), i = 1,...,r,$ and assume that the intersection graph $G(\phi)$ is connected. Then $F$ is connected.
\item In particular, if $(X,x)$ is reduced then $F$ is connected.
\end{enumerate}
Here $G(\phi)$ is the graph with vertices $i = 1,...r$, and we join $i \neq j$ by an edge iff  there exist points $y \in X \cap  \sX_i \cap \sX_j$  arbitrary close to $x$ ($y=x$ being allowed) such that $(X, y)$ is reduced. 
\end{theorem}

In (1) we need in fact only that at least one irreducible component of $X$ is generically reduced.
We do not assume that $\phi$ is flat, but this is practically irrelevant.
Since flatness means that no irreducible component of $(\sX,x)$ is mapped to 0,  
the irreducible components which are mapped to 0 do not contribute to the Milnor fibre and $F$ is the same as the restriction of $\phi$ to the other components, which is flat.
\medskip

The proof of Theorem \ref{thm:connected} is somewhat involved and uses the monodromy and the following general fibration theorem of L\^e D\~ung Tr\'ang (cf. \cite{Le77}, and \cite{Le76} for a detailed account).

\begin{theorem}[L\^e]\label{thm:le}
Let $\phi : \sX \to D$  be a good representative of $\phi : (\sX, x) \to (\C, 0)$.
Then
$$ \phi:\sX \setminus \sX_0 \to D \setminus \{0\} $$
is a topological fibre bundle with fibre $F$. 
\end{theorem}
This theorem has been known before in many special cases, all  generalizing
 Milnor's famous fibration theorem \cite{Mi68} for smoothings of an isolated hypersurface singularity.
\medskip

If $(X,x)$ has an isolated singularity, then $\partial X$
 is a real manifold diffeomorphic to  $\partial \sX_t$ for all $t\in S$ (by the Ehresmann fibration theorem, see e.g. \cite{Sea06}). 
Hence $\partial F$ is independent of the deformation $\phi$. If moreover  $\phi$ is a smoothing then $F$ is a Stein manifold, and  $\partial X$ can be filled by a complex Stein manifold. This imposes the following topological condition on the smoothability of $(X,x)$ (cf. \cite[2.2 Corollary]{GS83}), which is a strengthening of \cite{Ha74}, who proved an analogous result for cohomology instead of homotopy.

\begin{theorem}[Greuel, Steenbrink]\label{thm:topsmcond}
Let $(X,x)$ be an isolated singularity of pure dimension $n$. If $( X,x$) is smoothable, then $$\pi_i(X \setminus \{x\}) = 0$$
for $0 \leq i \leq min \{n - 2, n - codim(X,x)\}$.
\end{theorem}

This result and a local Lefschetz-Barth theorem of Hamm \cite {Ham81} is used in the proof of the following result about the homotopy groups of the Milnor fibre (cf. \cite[Theorem 1]{GS83}). Since $F$ is Stein,
$\pi_i(F) = 0$ for $i> \dim(X,x)$, and for the other homotopy groups we have:

\begin{theorem}[Greuel, Steenbrink]\label{thm:topMiFibre}
Let  $F$ be the Milnor fiber of a smoothing of a pure $n$-dimensional isolated singularity $(X,x)$. Then 
$$\pi_i(F) = 0 \text{ for } 0 \leq i \leq n - codim(X,x).$$
\end{theorem}

The following theorem (\cite[Theorem 2]{GS83}) is the main result in \cite{GS83}.
It was conjectured by J. Wahl, who proved it when $(X,x)$ is
weighted homogeneous and the smoothing has negative weight (cf. \cite{Wa82}).

\begin{theorem}[Greuel, Steenbrink]\label{thm:GS}
Let $F$ be the Milnor fiber of a smoothing of a normal isolated singularity. Then 
the first Betti number $b_1(F) := \dim_\C H^1(F,\C)= 0$.
\end{theorem}

The proof considers a good representative 
$\phi: \sX \to D$ of the smoothing of $(X,x)$ and uses a resolution 
$\pi : \widetilde\sX \to \sX$
of singularities of $\sX$, such that  $E = \tilde \phi^{-1}(x)$, $\tilde \phi =\phi \circ \pi$, is a divisor with normal crossings. Using the normality of $X$, 
it is proved that $H^1(E,\Z) = H^1(\widetilde \sX,\Z) = H^1(\widetilde \sX,\ko_{\widetilde \sX})= 0$.
The hypercohomology sheaves
${\bf R}^p\tilde \phi_*  K^\bullet$, $K^\bullet = \Omega^\bullet_{ \widetilde \sX /D}(\text{log } E)$, of relative logarithmic differential  forms are coherent (by \cite{BG80}) 
and locally free (by \cite{Ste76})
and satisfy $b_1(F) = b_1(\widetilde \sX_t) = \dim_\C {\bf H}^1(E,K^\bullet \otimes \ko_E)$. A careful study of the 2nd spectral sequence of hypercohomology,  and using  $H^1(E,\C) =0$,
leads to the required result.
 \medskip
 
The following corollary is immediate:

\begin{corollary} Let $X$ be any compact complex space with at most isolated normal singularities and let $\phi: \sX \to (\C, 0)$, be any smoothing of $X$. Then
$b_1(\sX_t )$ is constant for any sufficiently small  $t \in \C \setminus \{0\}$.
\end{corollary} 

The example of Pinkham shows that $\pi_1(F)$ need not be zero in Theorem \ref{thm:GS}. It is shown in \cite{GS83} that the assumption ``normal'' is necessary with the following example.  Take a smooth n-dimensional projective variety 
$E \subset \P^{N - 1}$ and let $X \subset \C^N$ be the affine cone over $E$. Let $F_0 \subset \P^{N - 1}$ be any smooth connected hypersurface of degree $d$, which intersects $E$ in a smooth variety $E_0 = E \cap F_0$ and let $G_0$ be the
affine cone over $F_0$. Consider in $\C^N$ the smoothing of the hypersurface section $X_0 = X \cap G_0$ through the origin by "sweeping out" the hypersurface section away from the origin with $X_t$ the nearby (Milnor) fibre. It is shown that 
$b_1(X_t) \geq b_1(E)$ and hence $X_0$ cannot be normal if $b_1(E)\neq 0$ (plenty of such $E$ exist).
\medskip 

In general the following holds by \cite[4.2 Proposition]{GS83}.

\begin{proposition}
 A normal connected projective variety $E$ with $\dim(E) \geq 2$ admits a projective embedding with projectively normal hypersurface section iff $b_1(E)=0$.
 \end{proposition}

Theorem \ref{thm:GS} was generalized by van Straten in \cite{Str17} using the same method.

\begin{theorem}[van Straten]\label{thm:straten}
Let $\phi: \sX \to D$ be a good representative of a smoothing of a reduced equidimensional singularity 
$(X,x)$ and $F$ its Milnor fibre. Let $X^{[0]}$ denote the disjoint union of the irreducible components of $X$ and $\gamma :H^0 (X^{[0]}) \to Cl(\sX,x)$ be the map that associates to a divisor supported on $X$ its class in the local class group. Then $b_1(F)  \geq rank Ker(\gamma)-1$, with equality if $X$ is weakly normal.
\end{theorem}
\medskip
\subsection{Milnor Number versus Tjurina Number}\label{sec:mutau}

We will now review some results about the Milnor number $\mu(X,x)$, an important topological invariant of the singularity, and in particular its (in some sense mysterious) relation to the Tjurina number $\tau(X,x)$ (Definition \ref{def:inf def:1}), which is an analytic invariant. The Minor number is defined as follows.

\begin{definition} Let $(X,x)$ be an $n$-dimensional isolated singularity and 
$\phi: \sX \to D$ a good representative of a 1-parametric deformation of $(X,x)$.
The middle Bettti number of the Milnor fibre $F$ of $\phi$,
  $$ \mu_\phi := b_n(F) = \dim_\C H^n(F,\C),$$
is called the {\em Milnor number}\index{Milnor!number} of $\phi$. 
If $\mu_\phi$ is independent of the deformation and depends only on 
$(X,x)$, we denote it by  $\mu(X,x)$.
\end{definition}

By Lemma \ref{lem:strat} there are only finitely many Milnor numbers of  $(X,x)$.
In this section we consider only singularities (e.g. complete intersections) with a unique Milnor number.  If  $(X,x)$ is a complete intersection or a normal isolated singularity then there are only two non-vanishing Betti numbers 
($ b_0(F) =1$ and $ b_n(F)$)  (see Theorem \ref{thm:GS} for the normal surface case). In general there are more non-vanishing Betti numbers. 

\medskip

Consider first a {\em hypersurface singularity} $(X,x) = (V(f),x)$ with isolated singularity, $f:(\C^{n+1},x) \to (\C,0)$ a holomorphic map germ.
Then the Milnor fibre $F=f^{-1}(t)$ is an $n$-dimensional complex manifold, which is homotopy equivalent to a bouquet of $n$-dimensional real spheres (Milnor \cite{Mi68}). Therefore the homology groups $H^i(F,\Z)$ do all vanish except for $i=1,n$. The middle Betti number ($F$ has real dimension $2n$)
plays a special role, and  
		$$\mu(X,x):=b_n(F)=\dim_\C H^n(F,\C),$$
the number of these spheres, is the {\em Milnor number}\index{Milnor!number} of $(X,x)$ or of $f$. Milnor proved the algebraic formula
		$$ \mu(X,x) = \dim_\C \ko_{\C^{n+1},x}/\textstyle\big\langle \frac{\partial f}{\partial x_0},\dots,\frac{\partial f}{\partial x_n}\big\rangle. $$ 
		
If $(X,x)$ is an n-dimensional {\em isolated complete intersection singularity} (ICIS)\index{ICIS}, then the homotopy type of the Milnor fibre is also a bouquet of $n$-spheres (Hamm \cite[Satz 1.7]{Ham71}) and the number of these spheres is again  the  {\em Milnor number}\index{Milnor!number} of $(X,x)$ and denoted by $\mu(X,x).$
\medskip

Since the base space of the 
semiuniversal deformation $\phi : (\sX,x) \to (S,s)$ of an ICIS is smooth (cf. Theorem 
\ref{thm:exis vers def}) there is only one Milnor fibre (up to diffeomorphism). 
In fact, the semiuniversal deformation is given by a flat morphism  
$\phi : (\C^{n+k},x) \to (\C^k,0)$ such that for a good representative 
$\phi : \sX \to S$ the restriction 
$\phi : \sX \setminus \phi^{-1}(\Del(\phi)) \to S \setminus \Del(\phi)$ 
  is a $C^\infty -$ fibre bundle (by \cite{Mi68} for a hypersurface, and \cite[Satz 1.6]{Ham71} for an ICIS). Here $C(\phi)$  denotes the critical locus  and 
$\Del(\phi) = \phi(C(\phi))$ the discriminant of $\phi$.
\bigskip

Milnor's algebraic formula for  $\mu(X,x)$ has been generalized to complete intersections independently by the author \cite{Gre75} (announced 1973 in \cite{BGr75}) and L\^e D\~ung Tr\'ang \cite{Le74}. The following result (cf. \cite[Lemma 5.3]{Gre75}) is an important step in the proof and of independent interest in itself.

\begin{proposition}\label{pr:mu+}
Let $(X,x)$ be an $n$-dimensional ICIS, $n \geq 0$, and $\phi: (\sX,x) \to (\C,0)$ a deformation of $(X,x)$ with $(\sX,x)$ an ICIS. Then 
		$$\mu(\sX,x) + \mu(X,x) = \dim_\C \ko_{C(\phi),x},$$
with $\ko_{C(\phi),x}$ the local ring of the singular locus of $\phi$ at $x$.
\end{proposition}
If $(X,x) \subset (\C^{n+k},x)$ is  defined by  $f_1,...,f_{k}$ and $(\sX,x)$ by  
$f_1,...,f_{k-1}$ (i.e. $\phi = f_k | (\sX,x)$), then 
$$\ko_{C(\phi),x} := \ko_{\C^{n+k},x} / \langle f_1,...,f_{k-1}, k \text{-minors of }
 Jac(f_1,...,f_{k-1},f_k)\rangle,$$
   where  $Jac$ denotes the Jacobian matrix. We can choose the $f_i$ such that 
   $(X_i,x) = V(f_1,...,f_{i-1})$, $i=1,...,k$, is an ICIS.
   Applying Proposition \ref{pr:mu+} to $f_i :(X_i,x) \to (\C,0)$ we get
   
\begin{theorem}[Greuel, L\^e] \label{thm:grle}
$$\mu(X,x) = \sum_{i=1}^k (-1)^{k-i} \dim_C \ko_{C(f_i),x}.$$
\end{theorem}

The proofs in \cite{Le74} and  \cite {Gre75} are very different. While the first is topological the second is algebraic and uses the Poincar\'e complex 
$\Omega_{X,x}^\bullet$ of holomorphic differential forms and an index theorem of Malgrange. An important result in \cite[Proposition 5.1] {Gre75}, from which \ref{thm:grle} is deduced and which has been extended to Gorenstein curves, is the following.

\begin{theorem}[Greuel]\label{thm:mu}
Let $(X,x)$ be an $n$-dimensional ICIS. Then 
$$
\mu(X,x) = \biggl \{  
\begin{array}{lllll}
\dim_\C \Omega_{X,x}^n/ d\Omega_{X,x}^{n-1}& if &n>0 \\
				\dim_\C \ko_{X,x} -1  &if& n=0.
\end{array}
 $$
\end{theorem}

If $(X,x)$ is quasihomogeneous, then $\mu(X,x)$ 
can be expressed purely in terms of the weights and degrees of the defining weighted homogeneous polynomials (c.f. \cite{GH78}).
\medskip

For a hypersurface defined by $f \in \ko_{\C^{n+1},0}$ we have obviously the inequality  $\mu(f) \leq \tau(f)$, which follows from the formulas for $\mu$ and $\tau$.
By a theorem of Saito \cite{Sa71} we have $\mu(f) = \tau(f)$ iff $f$ is analytically equivalent to a weighted homogeneous polynomial. The same result 
was conjectured in  \cite{Gre80} for complete intersections although the relationship is not at all obvious  ($\tau$ is the dimension of a vector space while $\mu$ is an alternating sum). 
The final proof 
is due to the author \cite[Korollar 5.8]{Gre75} and \cite[3.1 Satz]{Gre80}, to Looijenga-Steenbrink \cite{LS85} and to Vosegaard \cite{Vo02}. 

\begin{theorem}\label{thm:glsv}
Let $(X,x) \subset (\C^m,x)$
be an ICIS of positive dimension, defined by $f_1,\dots,f_k$.  
\begin{enumerate}
\item {\em (Looijenga-Steenbrink)} $\mu(X,x) \geq \tau(X,x).$ 
\item  {\em (Greuel)} If $(X,x)$ is quasihomogeneous, then \\
$
\begin{array}{cllll}
\mu(X,x) & =& \tau(X,x) \ = \ \tau'(X,x) := \dim_\C  \ko_{C(X),x}, \text{ with}\\ 
 \ko_{C(X),x} &=& \dim_\C  \ko_{\C^{m},x} / \langle f_1,...,f_{k}, k \text{-minors of }Jac(f_1,...,f_k)\rangle.
\end{array}
$
\item  {\em(Vosegaard)} If $\mu(X,x) = \tau(X,x)$ then $(X,x)$ is quasihomogenous.
\end{enumerate}
\end{theorem}

Each item is hard to prove. (1)  was proved for  
$\dim(X,x)= 1$  or if \ $\partial X$ is a rational homology sphere in  \cite{Gre80}. 
(3) was also proved before  in special cases:
for $\dim(X,x)= 1$ by Greuel-Martin-Pfister in \cite{GMP85}, see also Corollary \ref{cor:gmp}
(for Gorenstein curves, in the irreducible case already in \cite{Gre82}), 
for $\dim(X,x)= 2$ by Wahl 
in \cite{Wa85} and for a purely elliptic ICIS of dimension $\geq 2$ by Vosegaard in 
\cite{Vo00}.

\begin{remark} Let $(X,x)$ be an ICIS of dimension $n \geq 1$.
It was shown in  \cite[Proposition 1.11]{Gre75} that  
$$\tau'(X,x) = \tau''(X,x) :=\dim_\C H^0_{\{x\}} (\Omega_{X,x}^n),$$ 
where $\Omega_{X,x}^{\bullet}$ is the Poincar\'e complex and  $H^0_{\{x\}}$ denotes
local cohomology (in this case the torsion submodule), and by 
 \cite[Proposition 5.7]{Gre75} we have
$$ \mu(X,x) =  \tau''(X,x) + \dim_\C H^n(\Omega_{X,x}^{\bullet} / H^0_{\{x\}} (\Omega_{X,x}^\bullet)).$$
In particular $ \tau'(X,x) \leq  \mu(X,x)$ with equality if $(X,x)$ is quasihomogenous
(there are however non-quasihomogeneous examples with $\mu = \tau'$ for $n\geq 2$).
Moreover,
$$ \tau(X,x) = \dim_\C \text{\em Ext}^1_{\ko_{X,x}} (\Omega_{X,x}^1, \ko_{X,x}) = 
\dim_\C H^{n-1}_{\{x\}} (\Omega_{X,x}^1),  $$
where the first equality is due to Tjurina \cite{Tj69} and the second follows from local
duality, see  \cite[1.2 Satz]{Gre80}.
In particular $\tau(X,x) = \tau'(X,x)$ if $n=1$.  In general no relation between 
$\tau(X,x)$ and $\tau'(X,x)$ is known.
Based on computations with {\sc Singular} we conjecture:
\end{remark}

\begin{conjecture}
$\tau(X,x) \leq \tau'(X,x)$.
\end{conjecture}
\medskip

If $(X,x)$ is not an ICIS, the base space of the semiuniversal deformation may have 
several irreducible components and  the topology of a nearby generic fibre depends in general on the component over which the fibre lives. This situation is studied in detail 
in the next section.

However, there are classes of singularities other than ICIS which have a smooth 
semiuniversal base space, like Cohen-Macauly singularities in codim 2 or Gorenstein in codim 3. For these there is a unique (up to homeomorphism) Milnor fiber (the generic fibre over the semiuniversal base) and a unique Milnor number, defined as the middle Betti number of the Milnor fibre (c.f. Definition \ref{def:scomp}). A special case are normal surface singularities in $(\C^4,0)$. They are smoothable, with a smooth semiuniversal base space and, if they are Gorenstein then they are already a complete intersection. For these Wahl offered in  \cite{Wa15} the following conjecture:
 
\begin{conjecture}[Wahl]
Let $(X,x)$ be a normal surface singularity in $(\C^4,0)$, not a complete intersection. Then  
$$\mu(X,x)  \geq \tau(X,x)  - 1,$$
with equality if and only if $(X,x)$ is quasihomogeneous.
\end{conjecture}
\bigskip

\subsection{Smoothing Components}\label{sec:smoothcomp}
For an isolated singularity $(X,x)$, which is not a complete intersection, the semiuniversal base space may have several irreducible components (see e.g. Example \ref{ex:ratnorm}) and the
Milnor fibre depends in general on the smoothing. It is interesting to know, which properties  are independent of the smoothing and depend only 
on $(X,x)$. 
Let 
   $$ \Psi : (\sY,y) \to (S,s)$$
be the semiuniversal deformation of $(X,x)$.
Recall that an irreducible component $(S',s)$ of $(S,s)$  is called a {\em smoothing component}, if the generic fibre $F$ over $S'$ is smooth. The diffeomorphism type of $F$ depends only on $(S',s)$ and $F$ is the Milnor fibre of this component.
\medskip

\begin{definition}\label{def:scomp}
Let  $(S',s)$  be a smoothing component of the isolated singularity $(X,x)$
and $\phi :  (\sX,x) \to (\C,0)$ a smoothing induced by a morphism 
$j : (\C,0) \to (S',s)$.
We denote the {\em dimension of the smoothing component}\index{smoothing!component!dimension of} by
 $$ e_\phi := \dim(S',s).$$
\end{definition}

If $(S,s)$ is smooth (e.g. for $(X,x)$ a complete intersection), then $e_\phi$ is independent of $\phi$ and equal to $\tau(X,x)$. 
\medskip

It was already mentioned after Theorem \ref{thm:loo} that the conjectures of Wahl in \cite{Wa82} have all been proved and that several of his statements there are now valid in greater generality. 
 Wahl considered in \cite{Wa81} normal surface singularities that are not complete intersections and compares the Milnor number of a smoothing with the dimension 
 of the smoothing component over which the smoothing occurs. In \cite[Conjecture 4.2]{Wa81} Wahl made the following interesting conjecture about $e_\phi $, which he proved  in special cases
and which was fully proved by the author and Looijenga in \cite{GL85}.

\begin{theorem}[Greuel, Looijenga]\label{thm:GL}
With the assumptions of Definition \ref{def:scomp} we have
    $$ \dim (S',s) = \dim_\C \coker (\Theta_{\sX/\C,x} \to \Theta_{X,x}),$$
    with $\Theta_{\sX/\C}$ the sheaf of relative derivations.
\end{theorem}
\noindent
We will comment on the proof at the end of this section. 
\medskip

The more recent paper  \cite{Wa15} is partly an updated survey on old results, but it contains also  new results and conjectures on normal surface singularities, which we like to recall.
Let $\phi :  (\sX,x) \to (\C,0)$ be a smoothing of a normal surface singularity $(X,x)$. Then  Wahl introduced another invariant, 
  $$ \alpha_\phi := \dim_\C \coker (\omega^*_{\sX/\C,x} \otimes \ko_{X,x}\to \omega^*_{X,x}),$$
with $\omega^*_{\sX/\C}$ the $\ko_\sX$-dual of the relative dualizing sheaf.
\bigskip
 
 Using Theorem \ref{thm:GS} and  \ref{thm:GL} Wahl relates $\mu_\phi, e_\phi$ and
$\alpha_\phi$ with  resolution invariants of $(X,x)$ and proves (\cite[Theorem 3.13]{Wa81} and \cite[Theorem 1.1]{Wa15} where Wahl denotes our $ e_\phi$ by $ \tau_\phi$\,\footnote{Wahl calls the dimension of a smoothing component the Tjurina number of the smoothing. We do not follow his terminology, as it is widely accepted to call the dimension of $T^1$ the Tjurina number (Definition \ref{def:inf def:1}) and, moreover, since Tjurina was the first to introduce $T^1$ (as an $Ext^1$).}
):
 
 \begin{theorem}[Wahl]\label{thm:Wa1}
 Let $\phi: (\sX,x) \to (\C,0)$ be a smoothing of a normal surface singularity $(X, x)$ and  $(Y, E) \to (X, x)$ a good resolution. Then (with $ \chi_T$ the topological Euler characteristic)
 \begin{enumerate}
\item $1+\mu_\phi =  \alpha_\phi + 13 h^1(\ko_Y) + \chi_T(E) - h^1 (- K_Y). $
\item  $e_\phi = 2\alpha_\phi + 12 h^1(\ko_Y) + h^1(\Theta_Y) - 2h^1(- K_Y).$
  \end{enumerate}
If (X, x) is Gorenstein, then $\alpha_\phi = 0$, so $\mu$ and $e$ are independent of the smoothing.
 \end{theorem}
 
For $(X, x)$ Gorenstein (but not necessary smoothable) denote by $\tilde \mu(X,x)$ resp. $\tilde e(X,x)$  the expressions for  $\mu_\phi$ resp. $e_\phi$ given by the above theorem (these are invariants of $(X, x)$ that may be negative if $(X, x)$ is not smoothable). Then by 
\cite[Theorem 3.13]{Wa81} and \cite[Theorem 1.2]{Wa15}

 \begin{theorem}[Wahl]
If $(X, x)$ is a Gorenstein surface singularity, then 
$\tilde \mu(X,x) - \tilde e(X,x) \geq 0$, with equality if and only if $(X,x)$ is quasihomogeneous.
 \end{theorem}
 
Wahl's new main conjecture in \cite{Wa15} uses 
the sheaf of logarithmic derivations $S_Y := (\Omega^1_Y(log(E))^*$ on the resolution.

\begin{conjecture}[Wahl]
For $(Y,E) \to (X,x)$ the minimal good resolution of a non--Gorenstein normal surface singularity $(X,x)$one has
$$h^1(\ko_Y) - h^1(S_Y) + h^1(\Lambda^2 S_Y) \geq 0,$$
with equality if and only if $(X,x)$ is quasi-homogeneous.
\end{conjecture}

The quasihomogeneous case is settled by Wahl himself (\cite[Theorem 3.3]{Wa15}).
 \begin{theorem}[Wahl]
If $(X,x)$ is quasihomogeneous and not Gorenstein, then 
$h^1(\ko_Y) - h^1(S_Y) + h^1(\Lambda^2 S_Y) = 0.$
 \end{theorem}
 
For further conjectures and results concerning smoothings of (special classes of) normal surface singularities, in particular for different formulas for $\mu_\phi$ and $e_\phi$, we refer to \cite{Wa15}.
 \bigskip
 
 At the end of this section, let us sketch the main steps in the proof of Theorem \ref{thm:GL} because the method is valid in a very general deformation theoretic setting (\cite[Section 3]{GL85}) and some aspects in our special situation are interesting in itself. For details of what follows we refer to \cite[Section 1 and 2]{GL85}.
 \medskip
 
 Let  $\phi :  (\sX,x) \to (\C,0)$ be a deformation of an isolated singularity  $(X,x)$.
 Deformations of a morphism were considered  in Definition \ref{def:var defs} and we consider now deformations of
 $\phi$.  In particular we have (Definition \ref{def:embedded defs})
 	$$ T^1_{(\sX,x)/(\C,0)} = \uDef_{(\sX,x)/(\C,0)}(T_{\eps}).$$
	
Let us give an explicit description of $T^1_{(\sX,x)/(\C,0)}$. 
We may assume that $\phi$ is
embedded (Corollary \ref{cor:defs can be embedded}), i.e.
$(\sX,x) \subset (\C^N \times \C,0)$ is an embedding such that its composite with the projection on $(\C, 0)$ yields $\phi$.  
Choose a good representative
$\phi :  \sX  \subset B \times D \to D$ and let  
 $\ki \subset \ko_ {B \times D}$ be the ideal sheaf defining $\sX$. Then we have an exact sequence of $\ko_\sX$-modules

\begin{equation*}
\ki/\ki^2\stackrel{\alp}{\longrightarrow}\Omega^1_{B \times D / D} 
\otimes_{\ko_{B \times D}}
\ko_\sX\longrightarrow\Omega^1_{\sX/D}\longrightarrow 0\,,
\end{equation*}

and dualizing with  $\kHom_{\ko_\sX} (-,\ko_\sX)$ we get the exact sequence
\begin{equation*}
0\longrightarrow\Theta_{\sX/D} \longrightarrow\Theta_{B\times D/D}\otimes_{\ko_{B\times D}}
\ko_\sX\stackrel{\bet}{\longrightarrow}  \kHom_{\ko_\sX} (\ki/\ki^2,\ko_\sX).
\end{equation*}

\begin{lemma} Setting
 $ T^1_{\sX/D} := \coker (\beta)$, we have an exact sequence of sheaves 
$$ 0\rightarrow\Theta_{\sX/D} \rightarrow\Theta_{B\times D/D}\otimes
\ko_\sX\stackrel{\bet}{\rightarrow}  \kHom_{\ko_\sX} (\ki/\ki^2,\ko_\sX) \rightarrow T^1_{\sX/D} \rightarrow 0.$$
Moreover, $T^1_{\sX/D,x} =T^1_{(\sX,x)/(\C,0)}$, i.e. an element of $T^1_{\sX/D,x}$ may be regarded as a deformation  $\Phi: (\sY,x) \to (D,0) \times T_\eps \to T_\eps$ of $\phi$, 
which induces $\phi: (\sX,x) \to (D,0) \to \{0\}$ by restricting to  
$\{0\} \subset T_\eps$ (up to isomorphism).
\end{lemma}

By Theorem \ref{thm:GL}  we have to consider the cokernel of the map $\Theta_{\sX/D,x} \to \Theta_{X,x}$, which appears in the following exact sequence.

\begin{lemma}
For any deformation of an isolated singularity as above
there is an exact sequence of $\ko_\sX$-modules,
$$
 0\rightarrow\Theta_{\sX/D}\stackrel{\phi}{\rightarrow}\Theta_{\sX/D} \rightarrow
 \Theta_X \rightarrow T^1_{\sX/D} \stackrel{\phi}\rightarrow T^1_{\sX/D}  
 \rightarrow T^1_{X}.
$$
\end{lemma}
\medskip

Now let $(S,s)$ be a complex germ and $j: (\C,0) \to (S,s)$ a morphism. We set
	$$ \Theta(j) := Der_\C(\ko_{S,s}, \ko) \text{ with }  \ko :=\ko_{\C,0} =\C\{t\}.$$
	
For $\zeta \in  \Theta(j)$ define $j_\zeta^* : \ko_{S,s} \to \ko_{\C,0}[\eps]/\eps^2$ by
$j_\zeta^* = j^* + \eps \zeta$. This ring map defines a morphism of complex germs $j_\zeta : (\C,0) \times T_\eps \to (S,s)$, which extends $j$. Hence $j_\zeta$ is a deformation of $j$. 
Applying the left-exact functor $Der_\C(\ko_{S,s} , -)$ to 
$0  \rightarrow \ko \stackrel{t}{\rightarrow} \ko \rightarrow \C \rightarrow 0$,
we have an exact sequence
$$ 
0 \rightarrow \Theta(j) \stackrel{t}{\rightarrow}\Theta(j) \rightarrow Der_\C(\ko_{S,s} , \C) \cong T_{S,s} \rightarrow 0,
$$	
where $T_{S,s}$ is the Zariski tangent space of $(S,s)$. Hence $\Theta(j)$
is a free $\ko_{\C,0}$ module of rank $ \dim_\C(\Theta(j) / t  \Theta(j))$ and 
$\Theta(j) / t  \Theta(j) =  \Theta(j) \otimes \C$ maps injectively onto a subspace $V$ of $T_{S,s}$.
Since $\Theta(j)$ is free, it follows
\begin{lemma}
$$
\begin{array}{cllll}
	\dim_\C V &=& \dim_\C \Theta(j) \otimes \C = rk_\ko \Theta(j) \\
		         &=& \text{dim of the Zariski tangent space of S}\\
		         & & \text{at the generic point of the image of } j.
\end{array}
$$
\end{lemma}

\begin{remark}
The following geometric interpretation of $V$ may be helpful. Embed $(S,s)$ in some $(\C^k,0)$. The Zariski tangent spaces $T_{S,j(t)} \subset \C^k$ fit together to form an analytic vector bundle over the punctured disk  $D_\del \setminus \{0\} \subset \C$ for sufficiently small $\del$. Then $V$ is limit of the Zariski tangent space $T_{S,j(t)}$ for $t \to 0$, taken in the Grassmannian of subspaces in $\C^k $. 
\end{remark}
\medskip

Now let $(S,s)$ be the base space of the semiuniversal deformation of $(X,x)$.
$\zeta \in  \Theta(j)$ determines a morphism $j_\zeta : (\C,0) \times T_\eps \to (S,s)$
extending $j$ as above and hence by pullback a deformation of $(X,x)$ over 
$(\C,0) \times T_\eps$ extending $\phi$. Thus we get an element of $T^1_{\sX/\C,x}$
and the corresponding map  $\Theta(j) \to T^1_{\sX/\C,x}$ is surjective by versality
(Definition \ref{def:versal}). We get

\begin {lemma} \noindent
Let $(S,s)$ be the base space of the semiuniversal deformation of $(X,x)$.
\begin{enumerate}
\item The natural $\ko$-homomorphism $\Theta(j) \to T^1_{\sX/\C,x}$ is onto.
\item The image of \, $ T^1_{\sX/\C,x} \to T^1_{X,x} $ coincides with the image of \,
$\Theta(j) \otimes \C \to T_{S,s}$ under the identification $T^1_{X,x} \cong T_{S,s}$. 
\end{enumerate}
\end{lemma}

Now we can derive easily the main result  from \cite{GL85}, which implies Wahl's conjecture (Theorem \ref{thm:GL}).
\begin{theorem}[Greuel, Looijenga] \label{thm.GL2}
Let $ \Psi : (\sY,y) \to (S,s)$ be the semiuniversal deformation of $(X,x)$ and 
$\phi: (\sX,x) \to (\C,0)$ induced by $j: (\C,0) \to (S,s)$. Then the dimension of the Zariski tangent space of $(S,s)$ at the generic point of the image of $j$ equals 
     $$rk _\ko T^1_{\sX/\C,x} + \dim_\C \coker (\Theta_{\sX/\C,x} \to \Theta_{X,x}).$$
\end{theorem}

In particular, if the generic point of the image of $j$ is nonsingular (e.g. if the fibre over the generic point is smooth or rigid), then this is the dimension of the irreducible component of $(S, s)$ to which $j$ maps.
\medskip

By openness of versality (Theorem \ref{thm:openness of versality}) $ \Psi : \sY \to S$ is a joint versal deformation of $(\sX_t,z)$
for any point $t \in S$ close to $s$ and any $z \in Sing(\sX_t)$. Therefore the germ 
$(S,t)$ is isomorphic to the cartesian product of the germs of the
 semiuniversal base spaces of $(\sX_t,z)$, which we denote by $(S_{\sX_t},t)$,  and an extra smooth factor (Proposition \ref{prop:versal -- semiuniversal})
 which we denote by $(T,t)$.
Since $\phi_* T^1_{\sX/\C}$ is free at a generic point $t$ in the image of $j$, we see that
	$$rk _\ko T^1_{\sX/\C,x}  = \sum_{z \in Sing(\sX_t)} \dim_\C T^1_{\sX_t,z} .$$
which is equal to the embedding dimension of $(S_{\sX_t},t)$ and differs from the embedding dimension of $(S,s)$ by $\dim (T,t)$. Theorem \ref{thm.GL2} implies therefore

\begin {corollary} \label{cor:extrafactor}
If $t$ is a generic point of the image of $j$ then
$$\dim(S,t) = 
\dim (S_{\sX_t},t) + \dim_\C \coker (\Theta_{\sX/\C,x} \to \Theta_{X,x}).$$
\end {corollary}
\medskip

The first general formula for the dimension of a smoothing component was obtained by Deligne \cite{De73} in the case of reduced curve singularities. Although his formula is local, Deligne's proof uses global methods. As an application of Corollary \ref{cor:extrafactor} a purely local proof of Deligne's formula was given in \cite{GL85}.

Let $(X,x)$ be a reduced curve singularity and $(\overline{X,x})$ its normalization.
Since  any derivation on $(X, x)$ lifts uniquely (in characteristic 0) to 
$(\overline{X,x})$ (cf. \cite{De73}), we get natural inclusions 
$\Theta_{X,x} \subset \Theta_{\overline{X,x}}$. 

\begin{theorem}[Deligne]\label{thm:deligne}
Any smoothing component of $(X,x)$ has dimension
   $$  3\delta(X,x) - \dim_\C \Theta_{\overline{X,x}}/\Theta_{X,x},$$
   with $\delta(X,x) = \dim_\C \ko_{\overline{X,x}}/\ko_{X,x}.$
\end{theorem}

Smoothing questions of curve singularities will be treated in teh next section.

\bigskip

 \subsection {Curve Singularities} \label{sec:curves}
 This section is about smoothing components of reduced curve singularities, the question of their smoothability, and the topology of the Milnor fibre.
 
At first glance the situation is different from that for singularities of bigger dimension.
 There are no topological obstructions to smoothability: a small perturbation of the parameterization of the curve singularity $(X,x)$ by linear terms parametrizes a smooth curve. However, the particular fiber of a family defined by ``deformation of the parametrization'' has in general embedded components and only the reduction of the special fiber agrees with $(X,x)$  (cf. \cite[1.2]{Ha74}).

A point of interest is Deligne's formula  (Theorem \ref{thm:deligne}) for the dimension of the smoothing component of a reduced curve singularity.  In the following we reformulate it by using more common invaraints of  $(X,x)$, and relate it to the Milnor number. 
Through our reformulation this formula is effectively computable and provides a useful criterion for non-smoothability of a curve singularity. This is shown by applying it to the examples of Pinkham. 
\medskip

Let us first fix the notations for the most common invariants of a reduced curve singularity $(X,x) \subset (\C^n,0)$. Let  
$\ko = \ko_{X,x}$ be the local ring, $\fm$ the  maximal ideal of $\ko$ and
         $$ n: (\overline{X,x}) \to (X,x)$$ 
the normalization with semi-local ring $\overline\ko = n_*\ko_{\overline{X,x}}$ and its Jacobson radical $\overline \fm$. 
Moreover, we set 
$$
\begin{array}{cllllll}
\sC &=& \text{Ann}_{\overline\ko}(\overline\ko/\ko) &=& \text{ conductor ideal},\\
\Omega &=& \Omega^1_{X,x} &=& \text{ holomorphic  (K\"ahler) 1-forms on } (X,x),\\
T\Omega &=& H^0_{\{x\}}(\Omega) &=& \text{ torsion submodule of } \Omega,\\
\overline\Omega &=& n_*\Omega^1_{\overline{X,x}} &=& \text{ holomorphic 1-forms on } (\overline{X,x}),\\
\omega &=& \omega_{X,x} &=& \text{ dualizing module of } (X,x),\\
\Theta &=& \Hom_\ko(\Omega,\ko) &=&  \text{ module of derivations on } (X,x),\\
\overline\Theta &=& \Hom_{\overline\ko}(\overline\Omega,\overline\ko) &=&  \text{ module of derivations on } (\overline{X,x}).\\
\end{array}
$$

Recall that $\omega = \Ext^{n-1}_{\ko_{\C^n,0}}(\ko,\Omega^n_{\C^n,0}) $ can be identified with $\omega = \{\gamma \in \overline\Omega\otimes K | \,\text{res}(\gamma) = 0 \}$, with $K$ the total ring of fractions of $\ko$.
Let $d: \ko \to \omega$ be defined  as the composition 
$$ \ko \xrightarrow{d} \Omega \xrightarrow{j}  \overline\Omega  \hookrightarrow \omega,$$
with $d: \ko \to \Omega$ the exterior derivation. 
We use the following classical numerical invariants (and $\mu$ from \cite{BG80}) of 
$(X,x)$:

$$
\begin{array}{clllllll}
\del &=& \del (X,x)&=& \dim_\C \overline \ko /\ko &=&  \text{ delta-invariant},\\
\mu &=& \mu (X,x)&=& \dim_\C \omega /d\ko &=&  \text{ Milnor number},\\
\lambda &=& \lambda(X,x) &=&  \dim_\C \omega/j\Omega  &=&  \text{ lambda-invariant},\\
\tau' &=& \tau' (X,x)&=& \dim_\C T\Omega &=&  \text{ length of the torsion},\\
\tau &=& \tau (X,x)&=& \dim_\C T^1_{X,x} &=&  \text{ Tjurina number},\\
m &=& m(X,x)&=& \dim_\C \overline\ko /\fm\overline\ko &=& \text{ multiplicity of }\ko,\\
r &=& r (X,x)&=& \dim_\C \overline\ko/\overline\fm &=&  \text{ number of branches},\\
t &=& t (X,x)&=& \dim_\C \omega/\fm \omega &=&  \text{ Cohen-Macaulay type},\\
c &=& c (X,x)&=& \dim_\C \overline\ko/\sC &=&  \text{ multiplicity of conductor}.\\
\end{array}
$$
Note that $j\Omega \cong \Omega/T\Omega$ and  $\tau' = \tau$ by Proposition \ref{prop:t1xx} (3) and duality (see  \cite[1.2 Satz]{Gre80}). Moreover, $(X,x)$ is Gorenstein iff $t=1$. 

We introduce
	$$m_1 = m_1 (X,x) := \dim_\C \overline \Theta /\Theta,$$
	$$ e = e(X, x) := 3\del - m_1,$$
and call $e$ the {\em Deligne number}\index{Deligne number} of $(X,x)$.
Recall two important relations among these invariants from \cite{BG80} and 
\cite{De73}.

\begin{theorem} \label{thm:DeGr}
Let $(X,x)$ be a reduced curve singularity.
\begin{enumerate}
\item {\em (Buchweitz, Greuel)} $\mu = 2 \del - r +1$.
\item {\em (Deligne)} $\dim E =e$ for every smoothing component $E$ of $(X,x)$.
\end{enumerate}
\end{theorem}

Before we consider the smoothing problem for curve singularities, let us recall
the main properties of $\mu$ from 
\cite[Theorem 4.2.2, 4.2.4]{BG80}.

\begin{theorem}[Buchweitz, Greuel]\label{thm:bg}
Let $\phi: \sX  \to D$ be a good representative of a deformation of the 
reduced curve singularity $(X,x)$ with Milnor fibre $F = \sX_t, t\ne 0$.
\begin{enumerate}
\item $\mu(X,x) - \mu(F) = b_1(F)$, 
\item $\mu(X,x) - \mu(F) \le \del(X,x) - \del(F) \le 0$.
\end{enumerate}
{\em Here 
$$\mu(F) := \sum_{y\in Sing(F)}\mu(F,y)$$ 
and similar for $\del$. Note that $\mu(X,x) =0$ iff $(X,x)$ is smooth.}
\end{theorem}

This and the following corollary show the in particular the topological meaning of $\mu$.

\begin{corollary}
\noindent
\begin{enumerate}
\item If $F$ is smooth, then $\mu(X,x) = b_1(F)$.
\item The following are equivalent:
\begin{enumerate}
\item [(i)] $\mu(\sX_t)$ is constant for all $t \in D$,
\item [(ii)] $\del(\sX_t)$ and $\sum_{y\in Sing(\sX_t)}(r(\sX_t,y)-1)$ are constant for $t \in D$,
\item [(iii)] $b_1(\sX_t) =0$ for all $t \in D$,
\item [(iv)] $\sX_t$ is contractible for all $t \in D$.

\end{enumerate}
\end{enumerate}
\end{corollary}

The Milnor number and the delta-invariant have been generalized to non-reduced curve singularities $(X,x)$ with an embedded component at $x$ in \cite{BrG90} with similar topological properties (however, the Milnor fibre need not be connected in this case). See also a generalization to
singularities of arbitrary dimension with $X\setminus \{x\}$ normal in \cite{Gre17}. 
\medskip

In order to get criteria for smoothability of reduced curve singularities we need further relations among the above invariants. It is convenient to work with fractional ideals, 
i.e. $\ko$-ideals in $K$.
$\overline\ko$ is the integral closure of $\ko$ in $K$ and choosing an 
$\overline\ko$-generator $\alpha$ of $\overline\Omega$ we get an 
isomorphism 
$$ \phi_\alpha : \overline\Omega\otimes K \xrightarrow{\cong} K,$$ 
with  $\quad  \phi_\alpha (\overline\Omega) = \overline\ko \subset K$. We denote the image of  $\omega$ in $K$ under $ \phi_\alpha $ again by $\omega$. 
It is shown in \cite[2.3 Lemma]{Gre82} 
that there exists an $ \overline\ko$-generator $f$ of $\sC$ 
such that $\tilde \omega : = f \omega$ satisfies
$$\ko \subset\tilde\omega\subset\overline\ko,  \quad   \sC = \tilde\omega:\overline\ko =\{h\in K |\, h\overline\ko  \subset \tilde\omega\},$$ 
with   $\ko =\tilde\omega$ iff $(X,x)$ is Gorenstein.
Moreover, if $\bar t$ is an $\overline \ko$-generator of $\overline \fm$ we set 
$\tilde\Omega := \bar t \cdot \phi_\alpha (j\Omega)$. 
Obviousliy $\tilde\Omega\cong \Omega/T\Omega$. 
For the proof of the following lemma see \cite[Section 2.4, 2.5]{Gre82}.

\begin{lemma}\noindent
\begin{enumerate}
\item $\dim_\C \omega/ \overline\Omega =  \del,$
\item $ \dim_\C\overline\ko/\tilde\omega =  c-\del,$
\item $ \dim_\C  \tilde\omega/\ko = 2\del -c,$
\item $ \dim_\C \tilde\omega \fm/\fm =  2\del-c - t+1,$
\end{enumerate}
\end{lemma}

These formulas are used to prove the following relation between $e$ and the other invariants  (\cite[2.5 Theorem]{Gre82}).

\begin{theorem}[Greuel] \label{thm:gre}
Let $(X,x)$ be a reduced curve singularity.
\begin{enumerate}
 \item $e = \mu + t -1 + \dim_\C \overline\ko/\tilde\omega \cdot \tilde\Omega 
 - \dim_\C \overline\ko/\tilde\omega \cdot \fm.$ 
\item $\del \le \del+t-1+m-r \le3\del-c+m-r\le \ e\ \le \mu+2\del-c \le3\del-r+1,$
 $\mu+2\del-c \le3\del-r < 3\del$ if $(X,x)$ is singular.
 \item If $(X,x)$ is quasihomogeneous then $e = \mu + t -1$. 
 \item Let $(X,x)$ be smoothable. If $(X,x)$ is quasihomogeneous then $\tau \ge \mu + t -1$ and equality holds iff $(X,x)$ is unobstructed.
\end{enumerate}
\end{theorem}

Note that $(4)$ gives a useful criterion to decide whether a smoothable curve is obstructed.

The above formulas yield an easy proof of the following theorem by Dimca and the author in \cite{DG18}.

\begin {theorem}[Dimca, Greuel] \label{thm:DG}
Let  $(X,x)$ be a reduced complete intersection curve singularity. Then
the following hold.
\begin {enumerate}
\item $\tau = \tau'  = \lambda \geq  \delta +m -r,$ 
\item $ \tau - \delta = \dim_\C (\bar{\Omega} / \Omega)$. In particular, one has the equality
$$\dim_\C (\bar{\Omega} / \Omega)=\delta - r + 1$$
if and only if the singularity $(X,x)$ is weighted homogeneous.
\item $\frac{1} {2} \,\mu< \tau$ if $(X,x)$ is not smooth.
\end {enumerate}
\end {theorem}  

Moreover in  \cite{DG18} the authors  pose the question

\begin{question} \label{q0}
Is it true that 
$$\frac{3}{4}\, \mu(X,0) <\tau(X,0)$$
for any isolated plane curve singularity?
\end{question} 
The answer to this question is positive for semi-quasi-homogeneous singularities 
$(X,s)$, see the recent preprint \cite{AB18}.
\bigskip 

For Gorenstein curves Theorem \ref{thm:gre} was complemented in \cite{GMP85}, giving a numerical characterization of quasi homogeneity for Gorenstein curves.
\begin{theorem}[Greuel, Martin, Pfister] \label{thm:gmp}
If $(X,x)$ is Gorenstein, then $e \le \mu$ with equality if and only if 
$(X,x)$ is quasihomogeneous.
\end{theorem}

The following corollary  generalizes Theorem \ref{thm:glsv} in dimension 1.

\begin{corollary}\label{cor:gmp}
If $(X,x)$ is Gorenstein and unobstructed, then $\tau \le \mu$ with equality if and only if $(X,x)$ is quasihomogeneous.
\end{corollary}

In \cite{GMP85} there is an example of a complete intersection with several branches, satisfying (a): $\tau = e<\mu$, 
(b): all branches are non-Gorenstein and non-quasihomogeneous and 
(c):  $\tau = e = \mu$ for each branch. This shows that the assumption ``Gorenstein'' in Theorem \ref{thm:gmp} is necessary that we cannot conclude from the branches to their union.

Some problems remain however open. 
\begin{problem}
Is the converse of Theorem \ref{thm:gre}(3) also true, i.e. does $e = \mu + t -1$ imply that $(X,x)$ is quasihomogeneous?
\end{problem}

\begin{problem}
Does the inequality  $e \le \mu + t -1$ always hold?
We conjecture this at least for $(X,x)$ smoothable.
\end{problem}

The answer to both problems is ``yes'' in the follwing cases: (a): $(X,x)$ is Gorenstein and 
(b):  $(X,x)$ is irreducible and the monomial curve of  $(X,x)$ has Cohen-Macaulay type $\le 2$.

\bigskip
We turn now to non-smoothable curves. We want to apply the following criterion (\cite[3.1 Proposition]{Gre82}) for non smoothability.

\begin{proposition}\label{pr:crit}
Let $\phi: \sX \to T$ a sufficiently small representative of a deformation with section $\sig: T \to \sX$ of a reduced curve singularity $(X,x)$. Assume:
\begin{enumerate}
\item[(i)] $(\sX_t,\sig(t))$ is singular and not isomorphic to $(X,x)$ for $t\ne 0$,
\item[(ii)] $T$ is irreducible and $\dim T \ge e(X,x)$.
\end{enumerate}
Then there is an analytic open dense subset $T_0 \subset T$, such that  
$(\sX_t,\sig(t))$ is not smoothable if $t \in T_0$.
\end{proposition}
The proof is easy: $\phi$ can be induced from the semiuniversal deformation with base $(S,0)$ by a map $\varphi: (T,0) \to (S,0)$. By $(i)$ $\varphi$ is finite and 
$\varphi(T)$ has dimension $\ge e = \dim E$ for every smoothing component $E$ of $S$. $\varphi$ cannot map $T$ to any $E$ since there are no smooth fibres over $T$.
$S$ may have smoothing components but the generic point of the image of $\phi$
cannot be smoothable by openness of versality.
\medskip

Let us finish with proving non-smootability for the curve singularity $L^n_r$  consisting of $r$ lines through the origin in  
 $\C^n$ in generic position. The following was proved in \cite[3.4 Theorem]{Gre82}, generalizing  Pinkham \cite[Theorem 11.10]{Pi74}, who proved it for the range $n < r< 2n$ by global methods.

\begin{theorem}[Pinkham; Greuel]
The curve singularity $L^n_r$ is not smoothable in the following ranges:
\begin{enumerate}
\item $n < r \le \binom{n+1}{2}$ { \em and } $(r - n - 2) (n - 5) \ge 7$,
\item $\binom{n+d-1}{d} < r \le \binom{n+d}{d+1}, d\ge2$ { \em and } 
        $r(n-3-3d)+3\binom{n+d}{d} \ge n^2-1$.
\end{enumerate}
E.g., $L^n_r$ is not smoothable if $r$ is  in the following intervals:
$$
\begin{array}{c |c |c |c |c |c}
 n & 6 & 7 & 8 & 9 & 10 \\
  \hline
 r & [15,21] & [13,30] & [13,72] & [13,193] & [14,419] \\
\end{array}
$$
\end{theorem}

The intervals for the case $n < r \le \binom{n+1}{2}$ have been slightly enlargesd by Stevens in \cite{St89}.
For $n=2$ and 3 arbitrary many lines are smoothable. It is also known that $n, n+1$ and $n + 2$ lines in $(\C^n,0)$ are always smoothable, but $L^n_{n + 3}$ not if $n \ge 12$ after the theorem. Note that for fixed $n$, the theorem shows the existence of non-srnoothable curves of lines only for $r$ within some finite interval (which is growing with $n$). Also, we obtain nothing for $n = 4, 5$.

\begin{problem}
Do there exist for fixed $n \ge 4$ non-smoothable curves $L^n_r$  if $r$ goes to infinity? It seems unlikely that this is not the case.
\end{problem}

The formula of Deligne cannot only be used to show that certain curve singularities are not smoothable but also that the semiuniversal base for a smoothable curve is not smooth, namely if $e < \tau $. Examples are $n$ (resp. $n+1$) general lines in $(\C^n,0)$, which are obstructed if $n\ge 4$ (resp. $n\ge 5$), cf. \cite{Gre82}.
\bigskip



\printindex

\end{document}